\documentclass[11pt, a4paper]{amsart}
\usepackage{color}
\def\RR{{\mathbb {R}}}

\def\J{{\mathcal{J}}}

\def\ep{\varepsilon}

\def\di{\displaystyle}

\def\R{{\mathbb {R}}}

\def\K{{\mathcal{K}}}
\def\L{{\mathcal{L}}}
\def\H{{\mathcal{H}}^{N-1}}

\def\pint{\operatorname {--\!\!\!\!\!\int\!\!\!\!\!--}}

\def\wg{W^{1,G}(\Omega)}
\def\wg0{W_0^{1,G}(\Omega)}

\newtheorem{teo}{Theorem}[section]
\newtheorem{lema}{Lemma}[section]
\newtheorem{defi}{Definition}[section]
\newtheorem{prop}{Proposition}[section]
\newtheorem{corol}{Corollary}[section]
\theoremstyle{definition}
\newtheorem{remark}{Remark}[section]

\renewcommand{\theequation}{\arabic{section}.\arabic{equation}}

\begin{document}

\title[A Free boundary problem for the $p(x)$- Laplacian] {A Free boundary problem for the $p(x)$- Laplacian}

\author[J.Fern\'andez Bonder \& S. Mart\'{i}nez \& N. Wolanski]
{Juli\'an Fern\'andez Bonder, Sandra Mart\'{\i}nez and Noemi
Wolanski}

\address{Departamento  de Matem\'atica, FCEyN
\hfill\break\indent UBA (1428) Buenos Aires, Argentina.}
\email{{ jfbonder@dm.uba.ar \\ smartin@dm.uba.ar \\
wolanski@dm.uba.ar} \hfill\break\indent {\em Web-page:}
 {\tt http://mate.dm.uba.ar/$\sim$jfbonder}
\hfill\break\indent {\em Web-page:}
 {\tt http://mate.dm.uba.ar/$\sim$wolanski}
 }

\thanks{Supported by ANPCyT PICT 2006-290, UBA X078 UBA X117 and CONICET PIP 5478/1438. All three authors are members of CONICET}


\begin{abstract}
We consider the  optimization problem of minimizing
$\int_{\Omega}|\nabla u|^{p(x)}+ \lambda \chi_{\{u>0\}}\, dx$ in
the class of functions $W^{1,p(\cdot)}(\Omega)$ with
$u-\varphi_0\in W_0^{1,p(\cdot)}(\Omega)$, for a given
$\varphi_0\geq 0$ and bounded. $W^{1,p(\cdot)}(\Omega)$ is the
class of weakly differentiable functions with $\int_\Omega |\nabla
u|^{p(x)}\,dx<\infty$. We prove that every solution $u$ is locally
Lipschitz continuous, that it is a  solution to a free boundary
problem and that the free boundary, $\Omega\cap\partial\{u>0\}$,
is a regular surface.
\end{abstract}

\maketitle

\section{Introduction}

In this paper we study a free boundary problem for the $p(x)-$Laplacian.
The $p(x)-$Laplacian, is defined as
\begin{equation}\label{px}
\Delta_{p(x)}u=\mbox{div}(|\nabla u(x)|^{p(x)-2}\nabla u).
\end{equation}
This operator extends the classical Laplacian ($p(x)\equiv 2$) and the so-called $p-$Laplacian ($p(x)\equiv p$ with $1<p<\infty$) and it has been recently used in image processing and in the modeling of electrorheological fluids.

For instance, Chen, Levin and Rao \cite{CLR} proposed the following model in image processing
$$
E(u)=\int_{\Omega}\frac{|\nabla u(x)|^{p(x)}}{p(x)}+|u(x)-I(x)|\, dx \to \mbox{min}
$$
where $p(x)$ is a function varying between $1$ and $2$. It is chosen $p(x)$
next to 1 where there is likely to be edges and next to 2 where it is
likely not to be edges. 

Observe that the Euler-Lagrange equation associated to $E$ is the $p(x)-$laplacian.

For the modeling of electrorheological fluids, see \cite{R}.

\medskip

On the other hand, a free boundary problem associated to the $p(x)-$Laplacian,
was studied in \cite{HK} namely, the obstacle problem.
In that paper, existence and H\"older continuity of minimizers was proved.
No further regularity was studied.

To our knowledge, no other free boundary problem associated to this operator
has been analyzed up to date.

This paper is devoted to the study of the so-called {\em Bernoulli free boundary problem}, that is
\begin{equation}\label{bernoulli-px}
\begin{cases}
\Delta_{p(x)}u = 0 & \mbox{in }\{u>0\}\\
u=0,\ |\nabla u| = \lambda^*(x) & \mbox{on }\partial\{u>0\}.
\end{cases}
\end{equation}
where $\lambda^*$ is a given function away from zero and infinity.

This free boundary problem, in the linear case $p(x)\equiv2$, was first studied by A. Beurling in \cite{B} for $N=2$.

Still in the linear setting and for $N\ge 2$, this problem was analyzed by H. Alt and L. Caffarelli in the seminal paper \cite{AC}. In that work, the authors prove existence of a weak solution by minimizing the functional
$$
u\mapsto \int_{\Omega} \frac{|\nabla u|^2}{2} + \frac{\big(\lambda^*(x)\big)^2}2\chi_{\{u>0\}}\, dx.
$$
Then, they prove local Lipschitz regularity of weak solutions and,
when $\lambda^*$ is $C^\alpha$, they prove the $C^{1,\alpha}$ regularity of the
free boundary up to some negligibly set of possible singularities.

Later, in \cite{ACF}, these results were extended to the quasilinear uniformly
elliptic case.

Problem \eqref{bernoulli-px} with $p(x)\equiv p$ was addressed in \cite{DP},
where the same approach was applied to obtain similar results in the $p-$Laplacian
case. In that paper, the authors had to deal with the problem of the degeneracy
or singularity of the underlying equation.

Recently, the method was further extended in \cite{MW1}, where
 this free boundary problem for operators with non-standard growth was treated in
 the setting of Orlicz spaces.

\medskip

The Bernoulli free boundary problem, appears in many different applications, such as limits of singular perturbation problems of interest in combustion theory
(see for instance, \cite{BCN, LW1, LW2}) fluid flow e.g. the problem of jets
(see for instance \cite{ACF2, ACF3}) and some shape optimization problems with a volume constrain (see for instance, \cite{AAC, ACS, FBMW1, FBRW, Le, M1, Te2}).

\medskip

In this work, in order to analyze the Bernoulli free boundary problem
\eqref{bernoulli-px}, we follow the same approach as in the previously
mentioned works and prove optimal regularity of solutions and $C^{1,\alpha}$ regularity
of their free boundaries.

So we consider the following  minimization problem: For
$\Omega$ a smooth bounded domain in $\R^N$ and $\varphi_0$ a
nonnegative function with $\varphi_0\in L^{\infty}(\Omega)$ and
$\int_\Omega |\nabla \varphi_0|^{p(x)}\,dx<\infty$, we consider
the problem of minimizing the functional,
\begin{equation}\label{problem}
\mathcal{J}(u)=\int_{\Omega}\frac{|\nabla u|^{p(x)}}{p(x)}+
\lambda(x)\chi_{\{u>0\}}\, dx
\end{equation}
in the class of functions
$$
\mathcal{K}=\Big\{v\in W^{1,p(x)}(\Omega)\colon v - \varphi_0\in W_0^{1,p(x)}(\Omega)\Big\}.
$$
For the definition of the variable exponent Sobolev spaces, see Appendix \ref{appA1}.

\bigskip

In order to state the main results of the paper, we need to introduce some notation and assumptions.

\subsection*{Assumptions on $p(x)$}
Throughout this work, we will assume that the function $p(x)$ verifies
\begin{equation}\label{pminmax}
1<p_{min}\le p(x)\le p_{max}<\infty,\qquad x\in\Omega
\end{equation}
When we are restricted to a ball $B_r$ we use $p_- = p_-(B_r)$ and $p_+ = p_+(B_r)$ to denote the infimum and the supremum of $p(x)$ over $B_r$.

We also assume that $p(x)$ is continuous up to the boundary and that it has a modulus of continuity $\omega:\R \to \R$, i.e. $|p(x)-p(y)|\leq \omega(|x-y|)$ if $|x-y|$ is small. At several stages it is necessary to assume that
$p$ is log-H\"older continuous. This is, $w(r)=C(\log\frac1r)^{-1}$.

The precise assumptions on the modulus of continuity $\omega$ will be clearly stated in each section.

For our main result we need  to assume further that $p(x)$ is Lipschitz continuous in $\Omega$. In that case, we denote by $L$ the Lipschitz constant of $p(x)$, namely, $\|\nabla p\|_{L^{\infty}(\Omega)}\leq L$

\subsection*{Assumptions on $\lambda(x)$}

In the firsts sections we will only need $\lambda(x)$ to be bounded away from zero and infinity. We denote $0< \lambda_1\le \lambda(x)\le \lambda_2<\infty$ for $x\in\Omega$.

We will assume in the last part that $\lambda(x)$ is H\"older continuous.

\subsection*{Main Results}
Our first result gives the existence of a minimizer and, under the assumption of Lipschitz continuity of $p(x)$ and that $p(x)\ge 2$, the Lipschitz regularity of minimizers.
\begin{teo}\label{lip}
We prove
\begin{itemize}
\item Assume that $p(x)$ is log-H\"older continuous. Then, there exists a  minimizer of $\J$ in $\K$. Any such minimizer $u$ is nonnegative, bounded and locally H\"older continuous.

\item Under the same assumptions, any minimizer $u$ is globally $p(x)-$subharmonic and
$$
\Delta_{p(x)}u = 0 \qquad \mbox{in } \{u>0\}.
$$

\item If $p\in C^{0,1}(\Omega)$ then, every minimizer is nondegenerate (see Corollary \ref{ucr}).

\item If moreover $p(x)\geq 2$ in $\Omega$, then $u$ belongs to $C^{0,1}_{loc}(\Omega)$
    .
\end{itemize}
\end{teo}

Our second result states that Lipschitz, nondegenerate minimizers of \eqref{problem} are weak solutions to \eqref{bernoulli-px}.

\begin{teo}\label{solucion.debil}
Assume that $p(x)$ is H\"older continuous and that $\lambda(x)$ is continuous.

Let $u$ be a nondegenerate, locally Lipschitz continuous minimizer of \eqref{problem}. Then, $\{u>0\}$ has finite perimeter locally in $\Omega$ and $\H(\partial\{u>0\}\setminus \partial_{red}\{u>0\})=0$.

Moreover, for every $x_0\in\partial_{red}\{u>0\}$, (this is, for every $x_0$ where there is an exterior unit normal $\nu(x_0)$ to $\partial\{u>0\}$ in the measure theoretic sense), $u$ has
the following asymptotic development,
\begin{equation}\label{asim}
u(x)=\lambda^*(x_0)\langle x-x_0,\nu(x_0)\rangle^-+o(|x-x_0|)\quad \mbox{ as }
x\rightarrow x_0
\end{equation}
where $\lambda^*(x)=\Big(\frac{p(x)}{p(x)-1}\,\lambda(x)\Big)^{1/p(x)}$.

Finally, for every $\phi\in C_0^{\infty}(\Omega)$, there holds
$$
-\int_{\{u>0\}} |\nabla u|^{p(x)-2}\nabla u\nabla \phi\, dx =
\int_{\partial_{red}\{u>0\}} (\lambda^*(x))^{p(x)-1}\phi\, d\H.
$$
That is, $u$ is a weak solution to \eqref{bernoulli-px} in the sense of distributions.
\end{teo}

Now, we arrive at the last result of the paper. Namely, the regularity of the free boundary $\partial\{u>0\}$ for Lipschitz minimizers of \eqref{problem}.

\begin{teo}\label{teo.main}
Let $p(x)$ be Lipschitz continuous, $\lambda(x)$ be H\"older continuous, and $u$ be a locally Lipschitz continuous minimizer of \eqref{problem}. Then, for $\H-$almost every point in the free boundary $\partial\{u>0\}$ there exists a neighborhood $V$ such that $V\cap\partial\{u>0\}$ is a $C^{1,\gamma}$ surface, for some $\gamma>0$.
\end{teo}

\subsection*{Technical comments}

We finish this introduction with some comments on the technical difficulties that we have encountered when dealing with the $p(x)-$Laplacian, highlighting the differences in the arguments with respect to the previous works on Bernoulli-type free boundary problems.

\begin{itemize}
\item As  mentioned in the Appendix \ref{appA1}, the log-H\"older continuity is a key ingredient in order to deal with variable exponent Sobolev spaces. For instance, up to date, this is the minimum requirement in order to have that $C^\infty$ be dense in $L^{p(\cdot)}$. See \cite{DHN}.

\item One fundamental tool in the analysis of this free boundary problem is the use of barriers. In order to construct barriers one has to look at the operator in non-divergence form. See \cite{ACF, DP, MW1}, etc. In the $p(x)-$Laplacian case, in order to write the equation in non-divergence form, one has to be able to compute the derivative of $p(x)$. Therefore, the assumption that $p(x)$ be Lipschitz continuous becomes natural. See Remark \ref{nodiver}.

\item Probably, the main technical difficulty that we have encountered is the fact that the class of $p(x)-$ harmonic functions is not invariant under the scaling $u(x) \mapsto u(tx)/k$ if $t\neq k$. In  \cite{DP, MW1} this invariance was used in a crucial way in the proof of the Lipschitz continuity of the solutions. See, for instance, the proof of Lemma 3.2 in \cite{DP}.

In order to overcome this difficulty we went back to the ideas in \cite{AC}, but we are left with the additional technical assumption that $p(x)\ge 2$ in order to get the Lipschitz continuity of the minimizers.

\item As for Harnack's inequality in the case of $p(x)-$harmonic functions, the inequality that holds is analogous to Harnack's inequality for the nonhomogeneous Laplace equation. Moreover, the constant in this inequality is not universal, but depends (in a nontrivial manner) on the $L^\infty$ norm of the solution. Nevertheless, the constant in Harnack's inequality remains invariant under homogenous scalings of a solution, see Remark \ref{harna2}.

\item We are not aware of the validity of the strong minimum principle for $p(x)-$harmonic functions (it does not come out of Harnack's inequality). This property was used at several stages in previous works. In our new arguments we use instead the nondegeneracy of minimizers (see Lemma \ref{prom1}), which is valid for any $p(x)>1$. With this property we can prove, for instance, Corollary \ref{otrocota} which is a crucial step to obtain the Lipschitz regularity of minimizers.

\item We believe that the hypothesis $p(x)\ge 2$ --that is needed in order to obtain the Lipschitz regularity of minimizers-- is purely technical. This assumption is only used in Lemma \ref{dx}. If one is able to prove this lemma for a general $p(x)$, this assumption can be eliminated.

\item There is another step where the hypothesis that $p(x)$ be Lipschitz is crucial.  Namely, in order to obtain the result on the regularity of the free boundary one needs a differential inequality for a function of the gradient. In this paper we  prove that if $u$ is $p(x)$-harmonic and if  $v=|\nabla u|$ is far from zero and infinity, then $v$ is a subsolution of an elliptic equation with principal part in divergence form (see Lemma \ref{ecugradu}). In order to prove this result we need to differentiate the equation, and therefore we need $p(x)$ to be Lipschitz.

\item As in \cite{AC}, the hypothesis  $\lambda(x)$  H\"{o}lder continuous is needed in the proof of the regularity of the free boundary in Section 8. Note that this is a natural assumption if one expects the $C^{1,\alpha}$ regularity of the free boundary to imply the $C^\alpha$ continuity of $\nabla u$  up to $\partial\{u>0\}$.
\end{itemize}

\medskip

\subsection*{Outline of the paper}

First, in Section 2, under the assumption of log-H\"older continuity of $p(x)$, by using standard variational arguments, we prove  the existence of a minimizer for $\J$ in the class $\K$. Then, we show that every minimizer is $p(x)-$subharmonic and bounded.

In Section 3, we analyze the regularity properties of minimizers and prove, under minimal assumptions on $p(x)$, that minimizers are H\"older continuous (Theorem \ref{cotaalpha}). As a consequence, we deduce that $u$ is $p(x)-$harmonic in $\{u>0\}$.

In Section 4, we further analyze the regularity of minimizers. This time, under the assumption that $p(x)$ is (locally) Lipschitz continuous, we first prove that $u$ is nondegenerate near a free boundary point (Corollary \ref{prom12}). Then, assuming further that $p(x)\ge2$, we prove that $u$ is locally Lipschitz continuous (Theorem \ref{Lip}). This completes the proof of Theorem \ref{lip}.

In Sections 5, 6 and 7 we assume that $u$ is a Lipschitz non-degenerate minimizer and that $p$ is locally H\"older continuous.

In Section 5, we begin the proof of Theorem \ref{solucion.debil} and show the positive density of $\{u>0\}$ and $\{u=0\}$ at  every free boundary point (Theorem \ref{densprop}).

In Section 6 we study the measure $\Lambda = \Delta_{p(x)}u$ and prove that it is absolutely continuous with respect to $\H\lfloor\partial\{u>0\}$. Then, we deduce that almost every point on the free boundary belongs to the reduced free boundary (Lemma \ref{redcasitodo}).

In Section 7, we finish the proof of Theorem \ref{solucion.debil} by proving the asymptotic development of $u$ near a free boundary point in the reduced boundary.

Finally, Section 8 is devoted to the proof of Theorem \ref{teo.main}.

We finish this paper with a couple of appendices with some previous and some new results about $p(x)-$harmonic and subharmonic functions, that can be of independent interest.



\section{The minimization problem}

In this section we look for minimizers of the  functional
$\mathcal{J}$.
 We begin by discussing the existence of extremals. Next, we prove
 that any minimizer is a subsolution to the equation $\mathcal L u=0$
 and finally, we  prove that $0\le u\le \mbox{sup\,}\varphi_0$.

\begin{teo}\label{existencia} Let $p\in C(\overline\Omega)$ and $0<\lambda_1\le \lambda(x)\le\lambda_2<\infty$.
If $\mathcal{J}(\varphi_0)<\infty$ there exists a minimizer
of $\mathcal{J}$.
\end{teo}

\begin{proof}
The proof of existence is similar to the one in \cite{MW1}. Since  we are dealing with the Sobolev variable exponent, we write it down for the
reader's convenience.

Take  a minimizing sequence $(u_n)\subset\K$, then
$\mathcal{J}(u_n)$ is bounded, so $\int_{\Omega}|\nabla u_n|^{p(x)}$
and $|\{u_n>0\}|$ are bounded. As  $u_n = \varphi_0$ in
$\partial\Omega$, we have by Remark~\ref{equi} that $\|\nabla
u_n-\nabla \varphi_0\|_{p(x)}\leq C$ and by Lemma~\ref{poinc} we
also have $\|u_n-\varphi_0\|_{p(x)}\leq C$. Therefore, by Theorem
\ref{ref}
 there exists a
subsequence (that we still call $u_n$) and a function $u_0\in
W^{1,p(\cdot)}(\Omega)$ such that
\begin{align*}
&  u_n \rightharpoonup  u_0 \quad \mbox{weakly in }
W^{1,p(\cdot)}(\Omega),
\end{align*}
and by Theorem \ref{imb}
\begin{align*}
&  u_n \rightharpoonup  u_0 \quad \mbox{weakly in }
W^{1,p_{min}}(\Omega).
\end{align*}
Now, by the compactness of the immersion
$W^{1,{p_{min}}}(\Omega)\hookrightarrow L^{p_{min}}(\Omega)$ we have that, for a subsequence that we still denote $u_n$,
$$
u_n \to u_0 \quad \mbox{a.e. } \Omega.
$$

As $\K$ is convex and closed, it is weakly closed, so $u_0\in \K$.

Moreover,
\begin{align*}
 \{u_0>0\} &\subset \liminf_{n\to\infty}\{u_n>0\} \quad\mbox{ so that}\\
 \chi_{\{u_0>0\}}&\le\liminf_{n\to\infty}\chi_{\{u_n>0\}}.
\end{align*}

On the other hand,
$$
 \int_{\Omega} \frac{|\nabla u_0|^{p(x)}}{p(x)}\, dx \le \liminf_{n\to\infty}
\int_{\Omega}\frac{|\nabla u_n|^{p(x)}}{p(x)}\, dx.
$$
In fact,
\begin{equation}\label{G}
\int_\Omega \frac{|\nabla u_n|^{p(x)}}{p(x)}\,dx\ge\int_\Omega \frac{|\nabla u_0|^{p(x)}}{p(x)}\,dx
+\int_\Omega |\nabla u_0|^{p(x)-2}\nabla u_0\cdot(\nabla u_n-\nabla u_0)\,dx.
\end{equation}

Recall that $\nabla u_n$ converges weakly to $\nabla u_0$ in
$L^{p(\cdot)}(\Omega)$. Now, since  $|\nabla u_0|^{p(x)-1}\in L^{p'(\cdot)}(\Omega)$, by
Theorem \ref{ref} and passing to the limit in \eqref{G} we get
$$
\liminf_{n\to\infty}\int_\Omega \frac{|\nabla u_n|^{p(x)}}{p(x)}\,dx\ge
\int_\Omega \frac{|\nabla u_0|^{p(x)}}{p(x)}\,dx.
$$

Hence
$$
\mathcal{J}(u_0)\le \liminf_{n\to\infty}\mathcal{J}(u_n) =
\inf_{v\in\K} \mathcal{J}(v).
$$
Therefore, $u_0$ is a minimizer of $\mathcal{J}$ in $\K$.
\medskip
\end{proof}

{\begin{lema}\label{subsol} Let $1<p(x)<\infty$ and $0\le \lambda(x)<\infty$. Let $u$ be a local
minimizer of
$\mathcal{J}$. Then, $u$ is  $p(x)-$subharmonic.
\end{lema}}
\begin{proof}
Let $\ep>0$ and $0\leq \xi\in C^{\infty}_0$. Using the minimality
of $u$  we have
\begin{align*}0 &\leq \frac{1}{\ep} (\mathcal{J}(u-\ep
\xi)-\mathcal{J}(u))\leq \frac{1}{\ep} \int_{\Omega}\frac{|\nabla
u-\ep \nabla \xi|^p}p-\frac{|\nabla u|^p}p\, dx\\& \leq -\int_{\Omega}
 |\nabla u-\ep \nabla \xi|^{p-2}\nabla (u-\ep \nabla
\xi) \nabla \xi\, dx\end{align*} and if
we take $\ep\rightarrow 0$ we obtain $$0\leq -\int_{\Omega}
|\nabla u|^{p-2}\nabla u\nabla \xi\, dx
$$
\end{proof}

\begin{lema}\label{princ-max} Let $p$ be log-H\"{o}lder continuous, $1<p(x)<\infty$, $0\le \lambda(x)<\infty$ and  $u$  a
minimizer of $\mathcal{J}$ in $\K$. Then,
$\di 0\leq u\leq \sup_{\Omega}\varphi_0$.
\end{lema}

\begin{proof}
The proof follows as in Lemma 1.5 in \cite{ACF} once we show that the functions $\min(M-u,0)$ and $\min(u,0)$ are in $W_0^{1,p(\cdot)}(\Omega)$, where $M=\sup_{\Omega} \varphi_0$.

But this fact follows from Corollary 3.6 and Theorem 3.7 in \cite{D}.
\end{proof}


\section{Holder continuity}

In this section we study the regularity of the minimizers of
${\mathcal J}$.

As a first step, we prove that minimizers are H\"older continuous
provided the function $p$ is log--H\"older continuous. We use ideas from
\cite{AM} and \cite{DP}.

\begin{teo}\label{cotaalpha1} Assume $p$ has modulus of continuity
$\omega(r)=C\log(\frac1r)^{-1}$.  Then, for every $0<\gamma_0<1$, there
exist $r_0(\gamma_0,p_{min})$ and $\rho_0=\rho_0(r_0,\gamma_0)$  such that, if $u$ is a minimizer of $\J$  in $B_{r_0}$ then,
$u\in C^{\gamma_0}(B_{\rho_0})$. Moreover, if $M>0$ is  such that
$\|u\|_{L^{\infty}(B_{r_0})}\leq M$ there exists
$C=C(N,p_{min},p_{max},\omega(r),\lambda_2,M,\gamma_0)$  such
that $\|u\|_{C^{\gamma_0}(\overline{B_{\rho_0}})}\leq C$.
\end{teo}
\begin{proof}
We will prove that  there exist $r_0$ and $\rho_0$ as in the statement such that, if $\rho\leq
\rho_0$  and $\|u\|_{L^{\infty}(B_{r_0})}\leq M$ then,
\begin{equation}\label{algosale} \Big(\pint_{B_\rho}|\nabla
u|^{p_-}\,dx\Big)^{1/{p_-}}\le C \rho^{\gamma_0-1}
\end{equation}
where $p_-=p_-(B_{r_0})$.

\medskip

In fact, let  $0<r\leq r_0$ and  $v$ be  the solution of
\begin{equation}\label{eculambda}
\Delta_{p(x)} v=0 \quad\mbox{in
}B_r, \qquad v-u\in W_0^{1,p}(B_r).\end{equation}

Let $u^s(x)=s u(x)+(1-s) v(x)$. By using that $v$ is a solution of
\eqref{eculambda} we get,
\begin{equation}\label{standard}\begin{aligned}
\int_{B_r} \frac{|\nabla u|^p}{p}- \frac{|\nabla v|^p}{p}\, dx&=
\int_0^1 \frac{ds}{s}\int_{B_r}\Big(|\nabla u^s|^{p-2}\nabla u^s-
|\nabla v|^{p-2}\nabla v\Big) \cdot \nabla(u^s-v)\, dx.
\end{aligned}
\end{equation}

 By a standard inequality (see Remark \ref{desip} ) we have that,
\begin{equation}\label{a1a22}\begin{aligned}
\int_{B_r} \frac{|\nabla u|^p}{p}- \frac{|\nabla v|^p}{p}\, dx \geq
&C\Big(\int_{B_r\cap\{p\geq 2\} } |\nabla u-\nabla v|^p \,
dx\\&+\int_{B_r\cap\{p<2\} } |\nabla u-\nabla v|^2\Big(|\nabla
u|+|\nabla v|\Big)^{p-2} \, dx\Big) ,\end{aligned}\end{equation}
where $C=C(p_{min},p_{max},N)$.

Therefore,  by the minimality of $u$, we have (if
$A_1=B_r\cap\{p<2\}$ and $A_2=B_r\cap\{p\geq 2\}$)
\begin{align}\label{rn2}&\int_{ A_2} |\nabla u-\nabla v|^p  \, dx\leq
C \lambda_2 r^N\\\label{rn3} &\int_{A_1}|\nabla u-\nabla
v|^2(|\nabla u|+|\nabla v|)^{p-2} \, dx\leq C\lambda_2 r^N
\end{align}
Let $\ep>0$. Take $\rho=r^{1+\ep}$  and suppose that $r^{\ep}\leq
1/2$. Take $\eta$ to be chosen later. Then, by Young inequality and
the definition of $A_1$ we obtain,
\begin{equation}\label{rn4}
\begin{aligned}
\int_{A_1\cap B_{\rho}} |\nabla
u-\nabla v|^p\, dx \leq& \frac{C}{\eta}\int_{A_1\cap
B_{r}}(|\nabla u|+|\nabla v|)^{p-2}|\nabla u-\nabla v|^2 \,
dx\\
&+ C\eta\int_{ B_{\rho}\cap A_1} (|\nabla u|+|\nabla v|)^p\,
dx\\
\leq& \frac{C}{\eta}  r^{N}+ C\eta\int_{ B_{\rho}\cap A_1}
(|\nabla u|+|\nabla v|)^p\, dx.
\end{aligned}
\end{equation}
Therefore, by \eqref{rn2} and \eqref{rn4}, we get,

\begin{equation}\label{rn51}\int_{B_{\rho}}|\nabla u-\nabla v|^p\, dx
 \leq \frac{C}{\eta} r^N+C\eta\int_{B_{\rho}\cap A_1}(|\nabla u|+|\nabla v|)^p\, dx.\end{equation}
where $C=C(\lambda_2,N,p_{min},p_{max}).$


Since, $|\nabla u|^q\leq C(|\nabla u-\nabla v|^q+|\nabla v|)^q)$,
for any $q>1$, we have by \eqref{rn51}, choosing $\eta$ small that

\begin{equation}\label{otra}\int_{B_{\rho}} |\nabla u|^p\, dx
 \leq {C} r^N+C\int_{B_{\rho}} |\nabla v|^p\, dx.\end{equation}
where $C=C(\lambda_2,N,p_{min},p_{max}).$

 On the other hand, we have by Lemma \ref{caccio}
\begin{equation}\label{parav1}
\int_{B_{r/2}} |\nabla v|^p\, dx\leq C \int_{B_{3r/4}}
\Big|\frac{v-\{v\}_{3r/4}}{r}\Big|^p\, dx.\end{equation} By the
regularity of
 solutions, (see
  \cite{AM}) we have that for any $0<\gamma<1$,
\begin{equation}\label{parav2}
|v-\{v\}_{\tfrac34 r}|\leq C(\gamma,\|v\|_{L^{\infty}(B_r)},\omega(r))\
r^{\gamma}\end{equation}

Therefore,
\begin{equation}\label{otra2}\int_{B_{\rho}} |\nabla u|^{p}\, dx
 \leq C r^N +
C(\gamma,\|v\|_{L^{\infty}(B_r)},\omega(r),p_{min},p_{max},N,\lambda_2)\
r^{N-(1-\gamma)p_+}.
\end{equation}

Since $v-u\in W_0^{1,p(\cdot)}(B_r)$, the same proof as that of Lemma \ref{princ-max}  shows that $\|v\|_{L^{\infty}(B_r)}\leq
\|u\|_{L^{\infty}(B_r)}$. On the other hand, since $u$ is a
subsolution, by comparison we have $0\leq u\leq v$ and then
\begin{equation}\label{cotalinfinito}
\|v\|_{L^{\infty}(B_r)}= \|u\|_{L^{\infty}(B_r)}\end{equation}

 This means that the constant $C$ depends on $p_{min},p_{max},\
N,\lambda_2,\gamma, \omega(r)$ and $\|u\|_{L^{\infty}(B_{r})}$.

Let $0<\gamma_0<1$ and let $\ep>0$ and $0<\gamma<1$ such that
$$\frac{N}{p_{min}}\frac{\ep}{1+\ep}+(1-\gamma)=1-\gamma_0.$$

Let $r_0>0$ such that
$$\frac{p_+(B_{r_0})}{p_-(B_{r_0})}\leq 1+\ep .$$

From now on we denote, $p_-=p_-(B_{r_0})$ and $p_+=p_+(B_{r_0})$.

Then,
$$\rho^{\frac{-(N\ep+(1-\gamma)p_+)}{1+\ep}}\leq
\rho^{\frac{-N\ep}{1+\ep}-(1-\gamma)p_-}.$$

Then we have by \eqref{otra2}, \eqref{cotalinfinito} and by our
election of $\gamma$ and $\rho$ that,
$$
\begin{aligned}\pint_{B_{\rho}} |\nabla u|^{p_-}\, dx&\leq
\pint_{B_{\rho}} |\nabla u|^{p}\,
dx+\frac{1}{|B_{\rho}|}\int_{B_{\rho}\cap\{|\nabla u|<1\}} |\nabla
u|^{p_-}\, dx \\&\leq \pint_{B_{\rho}} |\nabla u|^{p}\, dx+ 1\\
&\leq 1+C \Big(\frac{r}{\rho}\Big)^N +
C
r^{-(1-\gamma)p_+}\Big(\frac{r}{\rho}\Big)^N
\\ & \leq 1+Cr^{-\ep N}+C r^{-\ep N-(1-\gamma)p_+}= 1+C\rho^{\frac{-N\ep}{1+\ep}}+ C
 \rho^{\frac{-N\ep-p_+(1-\gamma)}{1+\ep}}\\ &\leq
 C\rho^{\frac{-N\ep-p_+(1-\gamma)}{1+\ep}}\leq C
 \rho^{\frac{-N\ep}{1+\ep}-(1-\gamma)p_-}
\end{aligned}
$$

where $ C(\gamma_0,N,M,\omega(r),\lambda_2,p_{min},p_{max})$.

Let $r_0$ as before for this choice of $\ep$ and small enough so
that $r_0^{\ep}\leq 1/2$. Then, if $\rho\leq \rho_0=r_0^{1+\ep}\le1$,
\begin{equation}\label{otra4}\pint_{B_{\rho}} |\nabla u|^{p_-}\,
dx\leq  C \rho^{-[\frac{N\ep}{(1+\ep)p_-}+(1-\gamma)]p_-}\leq  C
\rho^{-[\frac{N\ep}{(1+\ep)p_{min}}+(1-\gamma)]p_-}= C
\rho^{-(1-\gamma_0)p_-}
\end{equation}
This is, if $\rho\leq \rho_0=r_0^{1+\ep}$
$$\Big(\pint_{B_{\rho}} |\nabla u|^{p_-}\,
dx\Big)^{1/p_-}\leq  C \rho^{\gamma_0-1}.$$

Therefore \eqref{algosale} holds.

Applying Morrey's Theorem, see e.g. \cite{MZ}, Theorem 1.53, we
conclude that, $u\in C^{\gamma_0}(B_{\rho_0})$
and $\|u\|_{C^{\gamma_0}(\overline{B_{\rho_0}})}\leq {C}$ with $
C(\gamma_0,M,N,\omega(r),\lambda_2,p_{min},p_{max})$.

\medskip

\end{proof}

Thus, we have the following,
\begin{teo}\label{cotaalpha} Assume $p$ has modulus of continuity $\omega(r)=C\log(\frac1r)^{-1}$.
Then, for every $0<\gamma_0<1$, any minimizer $u$ belongs to
$C^{\gamma_0}(\Omega)$. Moreover,  let $\Omega'\subset\subset\Omega$ and
$M=\|u\|_{L^{\infty}(\Omega)}$. There exists,
$C=C(N,\Omega',p_{min},p_{max},\omega(r),\lambda_2,M,\gamma_0)$ such that
 $\|u\|_{C^{\gamma_0}(\overline{\Omega'})}\leq C$
\end{teo}

Then, we have that $u$ is continuous. Therefore, $\{u>0\}$ is open.
We can prove the following property for minimizers.

\begin{prop}\label{solucion} Assume $p$ has modulus of continuity $\omega(r)=C\log(\frac1r)^{-1}$. Let $u$ be a minimizer of
$\mathcal{J}$ in $\K$. Then, $u$ is  $p(x)$--harmonic in
$\{u>0\}$.
\end{prop}
\begin{proof}
Let $B\subset \{u>0\}$ be a ball and let $v$ such that
$$
\Delta_{p(x)}v=0 \quad \mbox{ in } B,\qquad
v-u\in W_0^{1,p}(B).
$$

Since
 $u>0$ in $B$ we get, proceeding as in \eqref{standard} and \eqref{a1a22},
 \begin{align*} 0&\geq
\int_{B}\frac{|\nabla u|^p}p-\frac{|\nabla v|^p}p\, dx+\lambda(x)
\chi_{B\cap\{u>0\}}-\lambda(x)\chi_{B\cap\{v>0\}}\ge \int_{B}\frac{|\nabla u|^p}p-\frac{|\nabla v|^p}p
\, dx\\&\geq C\Big(\int_{A_1} \big(|\nabla u|+|\nabla v|\big)^{p-2}
|\nabla u-\nabla
v|^2\, dx+ \int_{A_2} |\nabla u-\nabla v|^p \, dx.
\Big)\end{align*}

Therefore, $$ \int_{A_1} \big(|\nabla u|+\nabla v|\big)^{p-2}|\nabla u-\nabla v|^2\,
dx=0.$$
Thus, $\big(|\nabla u|+\nabla v|\big)^{p-2}|\nabla u-\nabla v|^2=0$ in $A_1$ and,
by the definition of
$A_1$, we conclude that $|\nabla u-\nabla v|=0$ in this set.

On the other hand, we also have
$$\int_{A_2} |\nabla u-\nabla v|^p\, dx=0$$ so that  $|\nabla u-\nabla
v|=0$ everywhere in $B$.

Hence, as $u-v\in W_0^{1,p}(B)$ we have that $u=v$. Thus,
$\Delta_{p(x)}u=0$ in $B$.

\end{proof}

\section{Lipschitz  continuity}

In this section we prove the Lipschitz continuity and the non degeneracy of the minimizers. We assume throughout this section  that $p(x)$ is Lipschitz continuous. We take ideas from \cite{AC}.

\begin{lema}\label{gamapromedio}
Let $p$ be Lipschitz continuous. Let $u$ be a minimizer in $B_r(x_0)
\subset\subset \Omega$ and $v$ a solution to
$$
\Delta_p v=0\quad\mbox{in }B_r(x_0),\quad \quad v-u\in W_0^{1,p}(B_r(x_0)).
$$
Then, there exist  $r_0=r_0(p_{max},p_{min},L,\|u\|_\infty)$ and $C=C(p_{max},p_{min},N)$ such that for every  $\ep>0$ there exists $M_{\ep}=M(\ep,p_{max},p_{min},L,\|u\|_\infty)$ so that
if $M\geq M_{\ep}$ and
$r\leq r_0$,
$$
\int_{B_r(x_0)} |\nabla (u-v)|^{p(x)}\, dx \geq C |B_r(x_0) \cap
\{u=0\}|\, M^{(1-\ep)p_-}.
$$
where $$M=\frac{1}{r}\sup_{B_{\frac34\,r}} u. $$

\end{lema}

\begin{proof}
First observe that if we take $u_r(x)=\frac{1}{r}u(x_0+rx)$, $v_r(x)=\frac{1}{r}v(x_0+rx)$  and $p_r(x)=p(x_0+rx)$
then,
\begin{align*}
& \int_{B_1} |\nabla (u_r-v_r)|^{p_r(x)}\, dx =r^{-N}
\int_{B_r(x_0)} |\nabla (u-v)|^{p(x)}\, dy,\\
& |B_1\cap \{u_r=0\}|=r^{-N}
|B_r(x_0)\cap \{u=0\}|, \\
& \sup_{B_{3/4}(0)} u_r=\frac{1}{r}\sup_{B_{\frac34\,r}(x_0)} u,
\end{align*}
and $\di\|\nabla p_r(x)\|_{\infty}=r\|\nabla
p(x_0+rx)\|_{\infty}$. Since $\|\nabla p_r(x)\|$ is small, if $r$
is small, we will assume that $r=1$ and $\|\nabla
p\|_{L^{\infty}(B_1)}\leq \delta$ with $\delta$ as small as needed (by taking $r_0$ small enough).

 So that, from now on we
assume that $x_0=0$ and $r=1$.

 For $|z|\le \frac12$ we consider the
change of variables from $B_1$ into itself such that $z$ becomes
the new origin. We call $u_z(x) = u \big( (1-|x|)z + x \big)$,
$v_z(x) = v\big((1-|x|)z+x\big)$, $p_z(x) =
p\big((1-|x|)z+x\big)$. Observe that this change of variables
leaves the boundary fixed. Define,
$$
r_\xi = \inf\Big\{ r \,/\, \frac18\le r\le 1\quad\mbox{and}\quad
u_z(r\xi)=0\Big\},
$$
if this set is nonempty.

Now, for almost every $\xi\in \partial B_1$ we have
\begin{equation}\label{cota-sup}
v_z(r_\xi \xi) = \int_{r_\xi}^1 \frac{d}{dr}(u_z-v_z)(r\xi)\,
dr\le \int_{r_\xi}^1 |\nabla(u_z-v_z)(r\xi)|\, dr.
\end{equation}

Let us assume that the following inequality holds:  There exist
$\delta_0>0$ and $M_\ep$ such that if $\delta\leq \delta_0$ and $M\geq M_{\ep}$,
\begin{equation}\label{v_z}
v_z(r_\xi \xi)\ge (1-r_\xi)M^{1-\ep}
\end{equation}
where $M=\sup_{B_{3/4}(0)} u.$

Let  $B=M^{1-\ep} $. Then, for any $\eta>0$ we have,
\begin{equation}\label{moco}
v_z(r_\xi \xi) \leq \int_{r_\xi}^1
\frac{|\nabla(u_z-v_z)(r\xi)|^{p_z(r\xi)}}{(\eta
B)^{p_z(r\xi)-1}p_z(r\xi)}\, dr+\int_{r_\xi}^1\frac{p_z(r\xi)-1}
{p_z(r\xi)} B\eta\, dr,
\end{equation}
and using \eqref{v_z} and \eqref{moco} with $\eta$ small we have,
\begin{equation}\label{moco2}
 \int_{r_\xi}^1 \frac{|\nabla(u_z-v_z)(r\xi)|^{p_z(r\xi)}}{
(B)^{ p_z(r\xi)-1}p_z(r\xi)}\, dr\geq C(p_+,p_-) (1-r_{\xi}) B.
\end{equation}
Therefore,
\begin{equation}\label{moco3}
 \int_{r_\xi}^1 |\nabla(u_z-v_z)(r\xi)|^{p_z(r\xi)}\, dr\geq C(p_+, p_-) (1-r_{\xi})
B^{p_-}
\end{equation}

 Then, using \eqref{moco3}, integrating
first over $\partial B_1$ and then over $|z|\leq 1/2$ we obtain as
in \cite{AC},
$$
\int_{B_1} |\nabla (u-v)|^{p(x)}\, dx \geq C |B_1 \cap \{u=0\}|
B^{p_-}.
$$

So we have the desired result.

\smallskip

Therefore, we only have to prove \eqref{v_z}. Observe that, since
$|z|\leq 1/2$, it is enough to prove that $v(x)\geq  M^{1-\ep}
(1-|x|)$ if $M$ is large enough.

If $|x|\leq 3/4$, by Remark
\ref{harna2} we have
$$v(x)\geq C_1 (\sup_{B_{3/4}(0)} v-3/4)\geq C_1
 (\sup_{B_{3/4}(0)} u-3/4)= C_1 (M-3/4)\geq \frac{C_1}{2} M$$
if $M\geq 2$, with $C_1$ depending on $p_{min},p_{max},L$ and $\|v\|_{\infty}$ (the bound
of $v$ before rescaling, see Remark \ref{harna2}, that equals the bound of $u$ before rescaling).

If $|x|\ge 3/4$ we prove by a comparison argument that
inequality \eqref{v_z} also holds. In fact, we know
$$
{v}\ge C_2 M \mbox{ in }\overline{B_{3/4}}.
$$
Take $w(x) = \theta M(e^{-\mu|x|^2} - e^{-\mu})$, where $\theta$
is such that $w\leq v$ on $\partial B_{3/4}$. Let $\mu_0$ and
$\ep_0$ as in  Lemma \ref{exp}. Then if $\mu\geq \mu_0$ and
$\delta\leq \ep_0$, there holds that
$$
\begin{cases}
C(\mu,\theta,M,p)\Delta_{p(x)}w\geq \bar{C}_1 (\mu-\bar{C}_2
\|\nabla p\|_{\infty}
|\log (\theta M)|)\geq \bar{C}_1 (\mu-\bar{C}_2 \delta |\log (\theta M)|)  & \mbox{in } B_1\setminus B_{3/4},\\
w \le C_2 \displaystyle M& \mbox{on } \partial B_{3/4},\\
w=0 & \mbox{on } \partial B_1.
\end{cases}
$$
Recall that $\theta=\di\frac{C_2}{e^{-\mu 9/16}-e^{-\mu}}$. Thus,
if $\mu$ is  large enough (depending on $C_2$), there holds that  $|\log \theta|\leq 2\mu$.

Therefore,
$$C(\mu,\theta,M,p)\Delta_{p(x)}w_{\mu}\geq \bar{C}_1 (\mu-\bar{C}_2 2\delta
\mu-\bar{C}_2  \delta |\log M|).$$
Let $\delta_0>0$. Assume further that
$\delta\le \delta_0<\frac{1}{4} (\bar{C}_2 )^{-1}$ and take $\mu=2
\bar{C}_2\delta_0 |\log M|$. Then,

\[
\begin{cases}
\Delta_{p(x)}w\geq 0  & \mbox{in } B_1\setminus B_{3/4},\\
w \le C_2 \displaystyle M& \mbox{on } \partial B_{3/4},\\
w=0 & \mbox{on } \partial B_1.
\end{cases}
\]

Thus, $w\leq v$ in $B_1\setminus\overline{B_{3/4}}$. Now,
$$
\begin{aligned}  w &\ge \theta e^{-\mu} \mu(1-|x|) M
\geq C_2 M e^{-7/16 \mu} \mu(1-|x|)\\
& \geq C_3 M^{1-C\delta_0} (1-|x|)
\delta_0 \log M= C_3 M^{1-C\delta_0} (1-|x|) \log M^{\delta_0}
\end{aligned}
$$
if $M\geq 1$.

Given $\ep>0$, assume further that $C\delta_0\leq
\ep$ and then, that $M$ is large enough (that is $M\geq M_{\ep}$) so that $C_3 \log M^{\delta_0}\geq 1$. Then,
$$
w\geq  M^{1-\ep} (1-|x|) \quad \mbox{ in }
B_1\setminus\overline{B_{3/4}},
$$
so that
$$
v\geq  M^{1-\ep} (1-|x|)\quad \mbox{ in }
B_1\setminus\overline{B_{3/4}}.
$$

Recalling the estimate inside the ball $B_{3/4}$ we get, as $M\ge1$,
$$v\geq C_2 M\geq  M^{1-\ep} (1-|x|) \quad \mbox{in } B_{3/4},$$
if $M$ is large enough, and \eqref{v_z} is proved.
\end{proof}

\begin{lema}\label{dx}
Let $p\in Lip(\Omega)$, $p\geq 2$ in $\Omega$, There exist
 $r_0=r_0(p_{max},p_{min},L,\|u\|_\infty)$ and $C_{max}=C_{max}(p_{max},p_{min},L,\lambda_2,M)$  such that if $r\leq r_0$, each
 local minimizer $u$ with $\|u\|_\infty\le M$ has the following property: If $B_{\frac34\,r}(x_0)\subset\subset\Omega$,
$$ \frac{1}{r} \sup_{B_{\frac 34\, r}(x_0)} u\geq C_{max} \quad \mbox{implies }\quad
\Delta_{p(x)}u=0 \quad \mbox{ in } B_r(x_0)$$
\end{lema}

\begin{proof}
Take $v$ as in the previous lemma. By a standard inequality we
have if $p(x)\geq 2$ (see \cite{DP}),
$$
\begin{aligned}
\lambda_2 |B_r(x_0)\cap\{u=0\}|&\geq \lambda_2 (|B_r(x_0)\cap\{v>0\}|-|B_r(x_0)\cap\{u>0\}|)\geq\\
 &\ge\int_{B_r(x_0)} \frac{|\nabla u|^{p(x)}}{p}-
\frac{|\nabla v|^{p(x)}}{p}\, dx\geq C \int_{B_r(x_0)} |\nabla u-\nabla
v|^{p(x)}\, dx.
\end{aligned}
$$

If $C_{max}$ is large enough,   by the previous lemma we get,
$$|\{u=0\}\cap B_r(x_0)\}|\geq C|\{u=0\}\cap B_r(x_0)\}|
\Big(\frac{1}{r} \sup_{B_{3/4 r}(x_0)}u\Big)^{p_-(1-\ep)}.$$ Therefore, if
$C_{max}$ is big enough we have that $|\{u=0\}\cap B_r\}|=0$ and
we obtain the desired result since
$$\int_{B_r} |\nabla u-\nabla v|^{p(x)}\, dx\leq C |\{u=0\}\cap B_r\}|$$
so that, $u=v$ in $B_r$.
\end{proof}

\begin{lema}\label{prom1}

Let $p\in Lip(\Omega)$. For any $0<\kappa< 1$ There exist $r_0$, $c_{\kappa}>0$ such that if $r\leq r_0$ and $B_r(x_0)\subset\subset\Omega$,
$$ \frac{1}{r} \sup_{B_{r}(x_0)} u\leq c_{\kappa} \quad \mbox{implies }\quad
u=0 \quad \mbox{ in } B_{\kappa r}(x_0).$$

Here $r_0$ depends on $\kappa,p_{min},p_{max}, L$  and $N$ and $c_{\kappa}$ depends also on
 $\lambda_1$.
\end{lema}

\begin{proof}

 We may suppose that $r=1$ and that $B_r$ is
centered at zero, (if not, we take the rescaled function
$\tilde{u}=\frac{u(x_0+r x)}{r}$). Moreover, by taking $r\leq r_0$
we may assume that $\|\nabla p\|_{L^{\infty}(B_1)}\leq \delta$.

Let $\ep:=\sup_{B_{\sqrt{\kappa}}} u .$  Choose $v$ as
$$v=\begin{cases} \frac{\ep}{c_{\mu}} (e^{-\mu |x|^2}-e^{-\mu
{\kappa}^2}) \quad &\mbox{ in } B_{\sqrt{\kappa}}\setminus
B_{\kappa},\\
0 \quad &\mbox{ in } B_{\kappa}, \end{cases}$$ where
$c_{\mu}=(e^{-\mu \kappa}-e^{-\mu \kappa^2})<0$.

 By  Lemma
\ref{exp} we have if $\mu$ is large enough,
\begin{align*}
\frac{-c_{\mu}}{\ep \mu} e^{\mu|x|^2}|\nabla v|^{2-p}\Delta_{p(x)}(-v)&\geq
C_1(\mu-C_2\|\nabla p\|_{\infty}\big|\log
\frac{\ep}{-c_\mu}\big|)\\&\ge C_1(\mu-C_2\|\nabla p\|_{\infty}|\log
{(-c_\mu)}|-C_2 \|\nabla p\|_{\infty} |\log \ep|).
\end{align*}
If $\mu\geq \frac{\log 2}{(\kappa(1-\kappa))}$ we have,
$$e^{-\mu\kappa}\leq e^{-\mu \kappa}(e^{-\mu \kappa(\kappa-1)}-1)=e^{-\mu \kappa^2}-e^{-\mu \kappa}= -c_{\mu}\leq
e^{-\mu\kappa^2}<1.$$

Then, $0>\log(-c_{\mu})\geq -\mu \kappa$.
Therefore,
\begin{align*}
\frac{-c_{\mu}}{\ep \mu} e^{\mu|x|^2}|\nabla v|^{2-p}\Delta_{p(x)}(-v)&\geq
C_1\big((1-C_2\kappa\|\nabla p\|_{\infty} ) \mu-C_2 \|\nabla
p\|_{\infty}|\log \ep|\big) .
\end{align*}
If $\delta\leq \frac{1}{2C_2\kappa}$ we have,
\begin{align*}
\frac{-c_{\mu}}{\ep \mu} e^{\mu|x|^2}|\nabla v|^{2-p}\Delta_{p(x)}(-v)&\geq
C_1\Big(\frac \mu2 -C_2\delta |\log \ep| \Big)\geq 0 ,
\end{align*}
if we choose $\mu\geq 2C_2 \delta |\log \ep|$.

Hence, if $r\leq r_0:= \frac{1}{2LC_2\kappa}$ so that $\|\nabla
p\|_{\infty}\leq \delta_0=\frac{1}{2C_2\kappa}$, and we take $\mu=2C_2 \delta_0
|\log \ep|=\kappa^{-1}|\log \ep|,$
$$\Delta_{p(x)}v<0 \quad \mbox{ in } B_{\sqrt{\kappa}}\setminus
\overline{B}_{\kappa}.$$

By construction $v\geq u$ on $\partial B_{\sqrt
\kappa}$. Thus, if we take
$$w=
\begin{cases}
\min(u,v)\quad&\mbox{in}\quad B_{\sqrt\kappa},\\
u&\mbox{in}\quad \Omega\setminus B_{\sqrt\kappa},
\end{cases}
$$
we find that $w$ is an admissible function for the minimizing problem. Thus,
using the convexity  we find that
$$
\begin{aligned}\di
&\int_{B_\kappa}\frac{|\nabla u|^{p(x)}}{p(x)}+\lambda\chi_{B_\kappa\cap\{u>0\}}\, dx\\
&\hskip1cm =\di \mathcal{J}(u)- \int_{\Omega\setminus
B_\kappa}\frac{|\nabla u|^{p(x)}}{p(x)} \,
dx+\int_{\Omega}(\di\lambda\chi_{B_\kappa\cap\{u>0\}}
 -\lambda\chi_{\Omega\cap\{u>0\}})\, dx\\
 &\hskip 1cm \le \mathcal{J}(w)-
 \int_{\Omega\setminus B_\kappa}\frac{|\nabla u|^{p(x)}}{p(x)} \, dx+\int_{\Omega}(\lambda\chi_{B_\kappa\cap\{u>0\}}
 -\lambda\chi_{\Omega\cap\{u>0\}})\, dx\\
&\hskip1cm\leq\int_{B_{\sqrt\kappa}\setminus B_\kappa} \frac{|\nabla
w|^{p(x)}}{p(x)}\, dx-\int_{B_{\sqrt\kappa}\setminus
B_\kappa}\frac{|\nabla u|^{p(x)}}{p(x)}\, dx
\\
 &\hskip1cm\le \int_{B_{\sqrt\kappa}\setminus B_\kappa}  |\nabla w|^{p(x)-2}  {\nabla w}
 (\nabla w-\nabla u)\,dx=- \int_{B_{\sqrt\kappa}\setminus B_\kappa} |\nabla w|^{p(x)-2} {\nabla w}
 \nabla (u-v)^+\,dx\\
 &\hskip1cm =- \int_{(B_{\sqrt\kappa}\setminus B_\kappa)}  |\nabla
 v|^{p(x)-2}
   {\nabla v}
 \nabla (u-v)^+\,dx
\end{aligned}
$$
and as $v$ is a classical supersolution we have,
$$\begin{aligned}
&\int_{B_\kappa}\frac{|\nabla
u|^{p(x)}}{p(x)}+\lambda\chi_{B_\kappa\cap\{u>0\}}\, dx\le
\int_{\partial B_{\kappa}}|\nabla v|^{p(x)-1}
 u \,
 \,d\mathcal{H}^{N-1}.
\end{aligned}$$
On the other hand, if $\mu\geq \frac{\log 2}{\kappa (1-\kappa)}$,
then $1-e^{-\mu \kappa(1-\kappa)}\geq 1/2$ and therefore,
$v$ satisfies
$$|\nabla v|_{|x|=\kappa}= \frac{2\kappa\ep\mu e^{-\mu \kappa^2}}{e^{-\mu\kappa^2}-e^{-\mu\kappa}}
=\frac{2\ep \kappa\mu}{1-e^{-\mu \kappa(1-\kappa)}}\leq 4\kappa \mu \ep=4\ep|\log \ep| $$

 Thus,
$$\begin{aligned}
&\int_{B_\kappa}\frac{|\nabla
u|^{p(x)}}{p(x)}+\lambda\chi_{B_\kappa\cap\{u>0\}}\, dx\le C(p) (\ep
|\log\ep|)^{p_--1} \int_{\partial B_{\kappa}}u \,d\mathcal{H}^{N-1}.
\end{aligned}$$
By Sobolev's trace inequality we have,
$$
\begin{aligned}
&\int_{\partial B_\kappa}u\le C(N,\kappa)\int_{B_\kappa}|\nabla u|+u\, dx\\
&\hskip1cm \le C(N,p,\kappa) \Big(\int_{B_\kappa} |\nabla
u|^{p(x)}+|{B_\kappa\cap\{u>0\}}|+\int_{B_\kappa}u\, dx
\Big)\\&\hskip1cm
 \le C(N,\kappa,p,\lambda_1) (1+\ep)\Big(\int_{B_\kappa}\frac{|\nabla
u|^{p(x)}}{p(x)}+\lambda_1 |\{u>0\}\cap B_\kappa| \Big)
\end{aligned}
$$
where in the last inequality we are using that $\int_{B_\kappa}u\,
dx\leq \ep |\{u>0\}\cap B_{\kappa}|$. Therefore,
$$\int_{B_\kappa}\frac{|\nabla u|^{p(x)}}{p(x)}\,
dx+\lambda_1|B_\kappa\cap\{u>0\}|\le  C (\ep
|\log\ep|)^{p_--1} \Big(\int_{B_\kappa}\frac{|\nabla
u|^{p(x)}}{p(x)}\, dx+\lambda_1|B_\kappa\cap\{u>0\}|\Big),$$ where
$C=C(N,\kappa,p,\lambda_1).$

 So that, if $\ep$ is small enough
$$\int_{B_\kappa}\frac{|\nabla u|^{p(x)}}{p(x)}\,
dx+\lambda_1|B_\kappa\cap\{u>0\}|=0.$$

In particular, $u=0$ in $B_{\kappa}$
and the result follows.
\end{proof}

As a corollary we have,

\begin{corol} \label{prom12} Let $p\in Lip(\Omega)$. There exist $r_0$, $C_{min}>0$ such that if $r\leq r_0$ and $B_r(x_0)\subset\subset\Omega$,
$$ \frac{1}{r} \sup_{B_{\frac34 r}(x_0)} u\leq C_{min} \quad \mbox{implies }\quad
u=0 \quad \mbox{ in } B_{r/2}(x_0).$$

Here $r_0$ depends on
$p_{min},p_{max}, L$  and $N$ and $C_{min}$ depends also on
 $\lambda_1$.
\end{corol}

Corollary \ref{prom12} states that, if $p$ is Lipschitz, then any minimizer is non-degenerate, i.e,

\begin{corol}\label{ucr}
Let $p\in Lip(\Omega)$. Let $D\subset\subset \Omega$, $x_0\in D\cap \partial\{u>0\}$. Then
$$\sup_{B_r(x_0)}u\geq C_{min}r,\quad \mbox{ if } r\leq r_0$$ where $C_{min}$ is the constant in
Corollary \ref{prom12} and $r_0$ depends also on $D$.
\end{corol}

\begin{corol}\label{otrocota}
Let $p\in Lip(\Omega)$ and $p\geq 2$. Let $D\subset\subset \Omega$, $x_0\in D\cap \partial\{u>0\}$. Then
$$\sup_{B_{\frac34\,r}(x_0)}u\leq C_{max}r,\quad \mbox{ if } r\leq r_0$$ where $C_{max}$ is the constant in Lemma
\ref{dx} and $r_0$ depends also on $D$.
\end{corol}
\begin{proof}
Assume by contradiction that the inequality is false. Then,  by
Lemma \ref{dx}, $\Delta_{p(x)} u=0$ in $B_r(x_0)$. Therefore, by the
regularity results in \cite{AM}, $\nabla u\in
C^{\alpha}(B_r(x_0))$ and, since $u\geq 0$ and $u(x_0)=0$, there holds that
$\nabla u (x_0)=0$. Thus, $|\nabla u(x)|\leq C\rho^{\alpha}$ in
$\overline{B_{\rho}}$ if $\rho\leq 3r/4 $. From here we have that $u(x)\leq
C \rho^{1+\alpha}$ in $\overline{B_{\rho}}$ if $\rho\leq 3r/4 $.

On the other hand, by
Corollary \ref{ucr}, $\frac43C_{min} \rho\leq\sup_{B_{\rho}(x_0)} u\leq C
\rho^{1+\alpha}$ if $\rho$ is small, which is a contradiction.
\end{proof}

Now we can prove the local Lipschitz continuity of minimizers of
$\mathcal{J}$ when $p\in Lip(\Omega)$ and $p\geq 2$.

\begin{teo}\label{Lip}
Let $p\in Lip(\Omega)$ and $p\geq 2$.
Let $u$ be a
minimizer of $\mathcal{J}$ in $\K$. Then, $u$ is locally Lipschitz
continuous in $\Omega$. Moreover, for any connected open subset
$D\subset\subset
 \Omega$ containing free boundary points, the Lipschitz constant of $u$ in $D$ is
estimated by a constant $C$ depending only on $N, p_{maz}, p_{min},L,
dist(D,\partial\Omega)$, $\|u\|_{L^\infty(\Omega)}$, $\||\nabla u|^{p(x)}\|_{L^{1}(\Omega)}$, $\lambda_1$ and $\lambda_2$.
\end{teo}
\begin{proof} The proof follows as in \cite{AC}, from Corollary \ref{otrocota} and the gradient estimate
$$
|\nabla u(y)|\le C\Big(1+\frac1r\sup_{B_r(y)}u\Big)^{\frac{p_{+}(B_r)}
{p_{-}(B_r)}}
$$
that holds if $\Delta_{p(x)}u=0$ in $B_r(y)$ (see Lemma \ref{gradient-estimate}).
\end{proof}

\section{Linear growth -- Positive density}
Throughout this section we will assume that $u$ is a locally Lipschitz, non-degenerate (i.e. satisfying the conclusions of Corollary \ref{ucr}) minimizer, and we also assume that $p$ is H\"older continuous.

\begin{teo}\label{densprop}
Suppose that $p$ is H\"older continuous, $u$ is Lipschitz with constant $C_{Lip}$ and  non-degenerate with constants $c$ and $r_0$. For any domain
$D\subset\subset \Omega$ there exists a constant $\tilde{c}$, with $0<\tilde{c}<1$
depending on $N,C_{Lip}, c, D$ and the H\"older modulus of continuity of $p$, such that, for any
minimizer $u$ and for every $B_r\subset \Omega$, centered at the
free boundary with $r\le r_0$ we have,
$$ \tilde{c} \leq \frac{|B_r\cap\{u>0\}|}{|B_r|} \leq 1-\tilde{c}$$
\end{teo}
\begin{proof}
First, by the non-degeneracy we have that there exists $y\in B_r$
such that $u(y)\ge cr$ so that,
$$\sup_{B_{\kappa r(y)}} u\geq u(y)\geq c r.$$
Therefore,
$$\frac{1}{\kappa r} \sup_{B_{\kappa r(y)}}u\geq
 \frac{c}{\kappa}.$$ Now, if $\kappa$ is
small enough, we have
$$\frac{1}{\kappa r} \sup_{B_{\kappa r(y)}}u>
 C_{Lip}.$$

Using the fact that $u$ is Lipschitz with constant $C_{Lip}$
we find that $u>0$ in $B_{\kappa r},$ where $\kappa=\kappa(C_{Lip},c).$
Thus,
$$\frac{|B_r \cap \{u>0\}|}{|B_r|}\geq \frac{|B_{\kappa
r}|}{|B_r|}= \kappa^N.$$

In order to prove the other inequality, we may assume that $r=1$ and the ball is centered at the origin. Let us suppose by contradiction
that there exists a sequence of minimizers $u_k$ in $B_1$, corresponding to powers $p_{min}\leq p_k(x)\leq p_{max}$ with the same H\"older modulus of continuity $\omega(r)$, $u_k$ Lipschitz with constant $C_{Lip}$ and non-degenerate with constant $c$
such
that, $0\in\partial\{u_k>0\}$ and $|\{u_k=0\}\cap
B_{1}|=\ep_k\rightarrow 0$. Let us take $v_k\in W^{1,p_k(x)}(B_{1/2})$
such that,
\begin{equation}\label{eqvk2}
\Delta_{p_k(x)} v_k=0\ \quad \mbox{ in } B_{1/2},\qquad
v_k-u_k\in W^{1,p_k(x)}_0(B_{1/2}).
\end{equation}

We have, by the arguments leading to \eqref{rn2}, \eqref{rn3},
$$
\begin{aligned}
& \int_{\{p_k\geq 2\}\cap B_{1/2}} |\nabla u_k-\nabla v_k|^{p_k(x)}\, dx\leq C \ep_k\quad
\mbox{ and }
\\
& \int_{\{p_k< 2\}\cap B_{1/2}}\big(|\nabla u_k|+|\nabla v_k|\big)^{p_k(x)-2} |\nabla u_k-\nabla v_k|^2\, dx\leq C \ep_k.
\end{aligned}
$$

Now, since $|\nabla u_k|\le C_{Lip}$ for every $k$,
$$
\begin{aligned}
&\int_{\{p_k< 2\}\cap B_{1/2}}\big(|\nabla u_k|+|\nabla v_k|\big)^{p_k(x)-2} |\nabla u_k-\nabla v_k|^2\, dx\ge\\
&\hskip1cm \int_{\{p_k< 2\}\cap B_{1/2}}\big(2|\nabla u_k|+|\nabla u_k-\nabla v_k|\big)^{p_k(x)-2} |\nabla u_k-\nabla v_k|^2\, dx\ge\\
&\hskip1cm C\int_{\{p_k< 2,|\nabla u_k-\nabla v_k|\le |\nabla u_k|\}\cap B_{1/2}}|\nabla u_k-\nabla v_k|^2\,dx+\\
&\hskip1cm C\int_{\{p_k< 2,|\nabla u_k-\nabla v_k|> |\nabla u_k|\}\cap B_{1/2}}|\nabla u_k-\nabla v_k|^{p_k(x)}\,dx
\end{aligned}
$$

On the other hand, using again that  $|\nabla u_k|\le C_{Lip}$ for every $k$, assuming that $p_-<2$,
$$
\begin{aligned}
& \int_{\{p_k< 2,|\nabla u_k-\nabla v_k|\le |\nabla u_k|\}\cap B_{1/2}}|\nabla
u_k-\nabla v_k|^{p_k(x)}\,dx\le\\
 &\hskip1cm C \int_{\{p_k< 2,|\nabla u_k-\nabla v_k|\le |\nabla u_k|\}\cap B_{1/2}}|\nabla
u_k-\nabla v_k|^{p_-}\,dx\le\\
&\hskip1cm C\Big(\int_{\{p_k< 2,|\nabla u_k-\nabla v_k|\le |\nabla u_k|\}\cap B_{1/2}}|\nabla
u_k-\nabla v_k|^{2}\,dx\Big)^{p_-/2}
\end{aligned}
$$

If $p_-\ge2$,
$$
\begin{aligned}
&\int_{\{p_k< 2,|\nabla u_k-\nabla v_k|\le |\nabla u_k|\}\cap B_{1/2}}|\nabla
u_k-\nabla v_k|^{p_k(x)}\,dx\le\\
&\hskip1cm C \int_{\{p_k< 2,|\nabla u_k-\nabla v_k|\le |\nabla u_k|\}\cap B_{1/2}}|\nabla
u_k-\nabla v_k|^{2}\,dx\le C \ep_k.
\end{aligned}
$$

Summing up we get,
\begin{equation}\label{wk}
\int_{B_{1/2}}|\nabla u_k-\nabla v_k|^{p_k(x)}\,dx\le C\max\{\ep_k,\ep_k^{p_-/2}\}.
\end{equation}

 On the other hand, since $\|u_k\|_\infty\le C_{Lip}$ and $\|v_k\|_{C^{1,\alpha}(B_\rho)}\leq
C(N,\rho, p_+,p_-,\omega(r),\|u_k\|_{L^{\infty}(B_{1/2})})$
(see \eqref{cotalinfinito} for the bound of $\|v_k\|_{L^\infty(B_{1/2})}$ and \cite{AM} for the regularity of $v_k$), there holds that, for a subsequence,
$v_k\to v_0$ and $\nabla v_k\to\nabla v_0$ uniformly on compact subsets of $B_{1/2}$.

Finally, since $ \|u_k\|_{Lip(B_{1/2})} \leq C_{Lip}$ we have, for a subsequence, $u_k\to u_0$ uniformly in $B_{1/2}$.

Let $w_k=u_k-v_k$. Then, $w_k\to u_0-v_0$ uniformly on compact subsets of $B_{1/2}$. Let us see that $u_0=v_0$.

In fact, by \eqref{wk} we have that
$\|\nabla w_k\|_{L^{p_k(x)}}\to 0$. Since $w_k\in W_0^{1,p_k(x)}(B_{1/2})$, by
Poincar\'e inequality we get that $\|w_k\|_{L^{p_k(x)}(B_{1/2})}\to0$.
By Theorem \ref{imb} there holds that $w_k\to0$ in $L^{p_-}(B_{1/2})$ and, for a subsequence, $w_k\to0$ almost everywhere. Thus, $u_0=v_0$.

Since, the $p_k$'s are uniformly H\"older continuous and are uniformly bounded, there exists $p_0$ such that (for a subsequence) $p_k\to p_0$ uniformly in $B_{1/2}$.

Now, recall that $v_k\to u_0$ in $C^1(B_{1/2})$.
Then, $\Delta_{p_0(x)} u_0=0$ in $B_{1/2}$.

As $u_k\to u_0$ uniformly in $B_{1/2}$ and are uniformly non-degenerate we get,
$\sup_{B_s} u_0\geq c s$ for $s$ small.
 But $u_0(0)=\lim u_k(0)=0$. By the same argument as that in Corollary \ref{otrocota} we arrive at a contradiction.
\end{proof}

\begin{remark}\label{rema}
Theorem  \ref{densprop} implies that the free boundary has Lebesgue
measure zero. In fact, in order to prove this statement, it is
enough to use the first inequality in Theorem \ref{densprop}, as
this estimate says that the set of Lebesgue points of
$\chi_{\{u>0\}}$ in $\partial\{u>0\}\cap D$ is empty. On the other
hand almost every point $x_0\in\partial\{u>0\}\cap D$ is a Lebesgue
point, therefore $|\partial\{u>0\}\cap D|=0$.
\end{remark}


\section{The measure $\Lambda=\Delta_{p(x)}u$}
We still assume that $u$ is a non-degenerate, locally Lipschitz minimizer.

In this section we assume that $p$ is  H\"older continuous. First, we prove that $\{u>0\}\cap\Omega$ is locally of
finite perimeter. Then, we study the measure $\Lambda=\Delta_{p(x)}u$
and prove that it is absolutely continuous with respect to the $\H$
measure restricted to the free boundary. This result gives rise to a
representation theorem for the measure $\Lambda$. Finally, we prove
that almost every point in the free boundary belongs to the reduced
free boundary.

\begin{teo}
For every $\varphi\in C_0^\infty(\Omega)$ such that ${\rm
supp}(\varphi)\subset\{u>0\}$,
\begin{equation}\label{ecuacion}
\int_\Omega  |\nabla u|^{p(x)-2}\nabla u\nabla \varphi=0.
\end{equation}
Moreover, the application
$$
\Lambda(\varphi):= -\int_\Omega |\nabla u|^{p(x)-2}\nabla u\nabla
\varphi\,dx
$$
from $C_0^\infty(\Omega)$ into $\R$ defines a nonnegative Radon
measure $\Lambda=\Delta_{p(x)} u$ with support on
$\Omega\cap\partial\{u>0\}$.
\end{teo}
\begin{proof}
We know that $u$ is  $p(x)-$subharmonic. Then, by the
Riesz Representation Theorem, there exists a nonnegative Radon
measure $\Lambda$, such that $\Delta_{p(x)} u=\Lambda$. And, as $\Delta_{p(x)}
u=0$ in $\{u>0\}$, for any $\varphi\in
C_0^\infty(\Omega\setminus\partial\{u>0\})$ there holds that $\Lambda(\varphi)=0$ and
the result follows.
\end{proof}

Now we want to prove that $\Omega\cap
\partial\{u>0\}$, has finite $N-1$ dimensional Hausdorff measure.
First, we need the following lemma,

\begin{lema}\label{proplim}
Let $u_k$ be a sequence of  minimizers in $B_1$ corresponding to powers $p_k(x)$ and coefficients $\lambda_k(x)$ with $1<p_{min}\le p_k(x)\le p_{max}<\infty$, $0<\lambda_1\le \lambda_k(x)\le\lambda_2<\infty$, and all the $p_k$'s with the same modulus of continuity $\omega(r)$. Assume
$u_k \rightarrow u_0$
uniformly in $B_1$,
$|\nabla u_k|\leq C_{Lip}$ in
$B_1$, and that the $u_k$'s are non-degenerate in $B_1$ with constants $c_0$ and $r_0$. Then,
\begin{enumerate}

\item $\partial\{u_k>0\}\to \partial\{u_0>0\}$ locally in
Hausdorff distance,

\medskip

\item $\chi_{\{u_k>0\}}\to \chi_{\{u_0>0\}}$ in $L^1(B_1)$,

\medskip

\item If $0\in \partial\{u_k>0\}$, then $0\in
\partial\{u_0>0\}$.
\end{enumerate}
\end{lema}
\begin{proof}The proof follows as in pp. 19--20 of \cite{ACF}.
\end{proof}

Now, we prove the following theorem,
\begin{teo}\label{rn-1}
For any domain $D\subset\subset \Omega$ there exist constants $c,
C$,
 depending on $N, C_{Lip}, c_0, r_0$, $p_{min}, p_{max}, \lambda_1,\lambda_2, \omega(r)$
 and $D$ such that, for any
minimizer $u$ with $|\nabla u|\le C_{Lip}$, non-degenerate with constants $c_0$ and $r_0$,
and for every $B_r\subset
\Omega$, centered on the  free boundary with $r\le r_0$, we have
$$ c r^{N-1}\leq \int_{B_r}d\Lambda \leq C r^{N-1}$$
\end{teo}

\begin{proof}
The ideas are similar to the ones for the case $p$ constant in \cite{DP}, with modifications similar to the ones in the proof of Theorem \ref{densprop}.
\end{proof}

Therefore, we have the following representation theorem

\begin{teo}[Representation Theorem] \label{repteo} Let $u$ be a non-degenerate, locally Lipschitz continuous
minimizer. Then,
\begin{enumerate}
\item $\H( D\cap\partial\{u>0\})<\infty$ for every
$D\subset\subset\Omega$.

\item There exists a Borel function $q_u$ such that
$$
\Delta_{p(x)} u= q_u \,\H \lfloor \partial\{u>0\}.
$$
i.e
$$-\int_\Omega |\nabla u|^{p(x)-2}\nabla
u\nabla \varphi\,dx=\int_{\Omega\cap
\partial\{u>0\}} q_u \varphi \,d\mathcal{H}^{N-1}, \qquad \forall\ \varphi \in \ C_0^\infty(\Omega).
$$
 \item For $D\subset\subset\Omega$ there are constants $0<c\le
C<\infty$ 
such that for $B_r(x)\subset D$ and $x\in \partial\{u>0\}$,
$$
c\le q_u(x)\le C,\quad c\,r^{N-1}\le
\H(B_r(x)\cap\partial\{u>0\})\le C\,r^{N-1}.
$$

\end{enumerate}
\end{teo}
\begin{proof}
It follows as in Theorem 4.5 in \cite{AC}.
\end{proof}

\begin{remark}\label{perfin}
As $u$ satisfies the conclusions of Theorem \ref{repteo}, the set
$\Omega\cap\{u>0\}$ has finite perimeter locally in $\Omega$ (see
\cite{F} 4.5.11). That is, $\mu_{u}:=-\nabla \chi _{\{u>0\}}$ is a
Borel measure, and the total variation $|\mu_u|$ is a Radon
measure. We define the reduced boundary as in \cite{F}, 4.5.5.
(see also \cite{EG}) by, $\partial_{red}\{u>0\}:=\{x\in
\Omega\cap\partial\{u>0\}/ |\nu_u(x)|=1\}$, where $\nu_u(x)$ is
a unit vector with
\begin{equation}\label{carac}
\int_{B_r(x)}|\chi_{\{u>0\}} -\chi_{\{y/\langle y-x,\nu_u(x)\rangle<0\}}|=o(r^{N})
\end{equation}
for $r\to 0$, if
such a vector exists, and $\nu_u(x)=0$ otherwise. By the results
in \cite{F} Theorem 4.5.6 we have,
$$\mu_u=\nu_u \H \lfloor \partial_{red}\{u>0\}.$$
\end{remark}

\begin{lema}\label{redcasitodo}
$\H(\partial\{u>0\}\setminus\partial_{\rm red}\{u>0\})=0.$
\end{lema}
\begin{proof}
This is a consequence of the density property of Theorem
\ref{densprop} and Theorem 4.5.6 (3) of \cite{F}.
\end{proof}


\section{Asymptotic development and identification of the  function $q_u$}

In this section we  still assume that $u$ is a non-degenerate, locally Lipschitz
continuous minimizer,  $p$ is H\" older continuous
and, moreover we  assume that $\lambda$ is continuous.

We  prove some properties of blow up sequences of minimizers and prove that
any limit of a blow up sequence is a minimizer. Then, we  find an asymptotic development of minimizers near points
in their reduced free boundary. Finally, we  identify the function
$q_u$ for almost every point in the reduced free boundary.

We first prove some properties of blow up sequences,
\begin{defi}
Let $B_{\rho_k}(x_k)\subset D\subset\subset\Omega$ be a sequence of balls with
$\rho_k\to 0$, $x_k\to x_0\in \Omega$ and $u(x_k)=0$. Let
$$
u_k(x):=\frac{1}{\rho_k} u(x_k+\rho_k x).
$$
We call $u_k$ a blow-up sequence with respect to
$B_{\rho_k}(x_k)$.
\end{defi}

Since $u$ is locally Lipschitz continuous, there exists a blow-up
limit $u_0:\R^N\to\R$ such that, for a subsequence,
\begin{align*}
& u_k\to u_0 \quad \mbox{in} \quad C^\alpha_{\rm loc}(\R^N)\quad
\mbox{for every}\quad 0<\alpha<1,\\
& \nabla u_k\to\nabla u_0\quad *-\mbox{weakly  in}\quad
L^\infty_{\rm loc}(\R^N),
\end{align*}
and $u_0$ is Lipschitz in $\RR^N$ with constant $C_{Lip}$.

\begin{lema}\label{propblowup}
If $u$ is a non-degenerate, locally Lipschitz continuous
minimizer then,
\begin{enumerate}

\item $\partial\{u_k>0\}\to \partial\{u_0>0\}$ locally in
Hausdorff distance,

\medskip

\item $\chi_{\{u_k>0\}}\to \chi_{\{u_0>0\}}$ in $L^1_{\rm
loc}(\R^N)$,

\medskip
\item $\nabla u_k\to\nabla u_0$ uniformly in compact subsets of
$\{u_0>0\}$,

\medskip

\item $\nabla u_k\to\nabla u_0$ a.e in $\Omega$,

\medskip

\item If $x_k\in \partial\{u>0\}$, then $0\in
\partial\{u_0>0\}$
\medskip

\item $\Delta_{p(x_0)} u_0=0$ in $\{u_0>0\}$

\medskip

\item $u_0$ is Lipschitz continuous and non-degenerate with the same constants $C_{Lip}$ and $c_0$  as $u$.
\end{enumerate}
\end{lema}

\begin{proof}
(1), (2) and (5) follow from Lemma \ref{proplim}. For the proof
of (3) and (4) we use that $\nabla u_k$ are uniformly H\"older continuous
in compact subsets of $\{u_0>0\}$ and ideas similar to those in  pp. 19-20 in \cite{ACF}. (6) follows from (3) and the fact that $\Delta_{P_k(x)}u_k=0$ in $\{u_k>0\}$ with $p_k(x)=p(x_0+\rho_kx)\to p(x_0)$ uniformly in compact sets of $\R^N$. (7) follows immediately from the uniform convergence of $u_k$ and the fact that they are all non-degenerate with constant $c_0$.
\end{proof}
\begin{lema}\label{minblow}
Let  $u$ be a non-degenerate, locally Lipschitz continuous
minimizer with $u(x_m)=0$, $x_m\rightarrow x_0\in \Omega$. Then, any blow up
limit $u_0$ respect to $B_{\rho_m}(x_m)$ is a
minimizer of
$\mathcal{J}$ corresponding to  $p\equiv p(x_0)$ and $\lambda\equiv\lambda(x_0)$ in any ball.
\end{lema}
\begin{proof}
See \cite{AC}.
\end{proof}

\medskip

In the sequel we will call $\lambda^*(x)=\Big(\frac{p(x)}{p(x)-1}\,\lambda(x)\Big)^{1/p(x)}$.

We have,

\begin{lema}\label{constant1}
Let $u$ be a non-degenerate, Lipschitz continuous, local minimizer in $\R^N$ corresponding to  $p(x)\equiv p_0$ and $\lambda(x)\equiv\lambda_0$, such that $u=\lambda_1\langle
x,\nu_0\rangle^-$ in $B_{R}$, with $R>0$, $0<\lambda_1<\infty$
and $\nu_0$ a unit vector. Then, $\lambda_1=\Big(\frac{p_0}{p_0-1}\,\lambda_0\Big)^{1/{p_0}}$.
\end{lema}

\begin{proof} See \cite{MW1}.
\end{proof}

\begin{lema}\label{distB}
Let $u$ be a locally Lipschitz, non-degenerate local minimizer in $B_1$ with power $p(x)$ H\"older and coefficient $\lambda(x)$ continuous. Let $x_0\in \partial\{u>0\}\cap B_1$ such that there exists a ball $B$ contained in $B_1\cap\{u=0\}$ touching $\partial\{u>0\}$ at the point $x_0$. Then,
$$
\limsup_{\stackrel{x\to x_0}{u(x)>0}}\frac{u(x)}{dist(x,B)}=\lambda^*(x_0).
$$
\end{lema}
\begin{proof} See, for instance \cite{MW1} for the idea of the proof. Here we use  Lemmas \ref{minblow} and \ref{constant1}.
\end{proof}

\begin{lema}\label{const}
Let $u\in \K$ be a minimizer. Then, for every  $x_0\in
\Omega\cap\partial\{u>0\}$
\begin{align}
& \limsup_{\stackrel{x\to x_0}{u(x)>0}} |\nabla u(x)| =
\lambda^*(x_0).
\end{align}
\end{lema}

\begin{proof} For the idea of the proof see, for instance \cite{MW1}. Here we use  Lemmas \ref{minblow} and \ref{constant1}.
\end{proof}

\begin{teo}\label{blow3}
Let $u$ be a minimizer, then for $\mathcal{H}^{N-1}-$a.e
$x_0\in\partial\{u>0\}$, the following properties hold,

 \begin{equation*}
 q_u(x_0)={\lambda^*(x_0)}^{p(x_0)-1}
\end{equation*}
and
\begin{equation}\label{asymp}
u(x)=\lambda^*(x_0)\langle x-x_0,\nu_u(x_0)\rangle^-+o(|x-x_0|)
\end{equation}
\end{teo}

\begin{proof} In order to prove \eqref{asymp} we follow the ideas of \cite{MW1}  using Lemma \ref{propblowup} items (6) and (7) and Lemmas \ref{minblow}, \ref{constant1} and \ref{const}.
\end{proof}

\medskip

\section{Regularity of the free boundary}

In this section we  assume that  $\lambda$ is H\"{o}lder continuous
 and $p$ Lipschitz with constant $L$, and therefore the corresponding $\lambda^*$
will also be H\"{o}lder continuous.
We denote by $C^*$ the constant of H\"older continuity of $\lambda^*$ and
 by $\alpha^*$ its H\"{o}lder exponent.

We prove the regularity of the free boundary of a
minimizer $u$ in a neighborhood of every ``flat'' free
boundary point. In particular, we prove the regularity in a
neighborhood  of every point in $\partial_{red}\{u>0\}$ where $u$
has the asymptotic development \eqref{asymp}. Then, if $u$ is a
minimizer, $\partial_{red}\{u>0\}$ is smooth and the remainder of
the free boundary has $\H-$  measure zero.

First, we  recall some definitions and then, we  point out the only
significant differences with the proofs in \cite{DP} with $p$ constant.
The rest of the proof of the regularity then follows as sections
6, 7, 8 and 9 of \cite{DP}.

\begin{remark}\label{T} In \cite{DP}, Sections 6, 7 and 8 the authors use the fact
that when $|\nabla u|\geq c$, $u$ satisfies  a linear
nondivergence uniformly elliptic equation, $Tu=0$. In our case we
have that  when $|\nabla u|\geq c$, $u$ is a solution of the
equation defined in \eqref{Tu}. As in those sections the
authors only use the fact that this operator is linear and
uniformly elliptic, then the results of those sections in
\cite{DP} extend to our case without any change.
\end{remark}

For the reader's convenience,  we  sketch here the proof of the
regularity of the free boundary by a series of steps  and we
write down the proofs in those cases in which we have to make
modifications.

\subsection{Flatness and nondegeneracy of the gradient}

\medskip

\begin{defi}[Flat free boundary points] Let $0<\sigma_+,
\sigma_-\le 1$ and $\tau>0$. We  say that $u$ is of class
$$
F(\sigma_+,\sigma_-; \tau)\quad \mbox{in}\quad B_\rho=B_\rho(0)\quad\mbox{with power }p(x)
$$
if $u$ is a local minimizer of $\J$ in $B_\rho$ with power $p(x)$,
\begin{enumerate}
\item $0\in \partial\{u>0\}$ and
$$
\begin{array}{ll}
u=0 & \mbox{for}\quad x_N\ge \sigma_+ \rho,\\
u(x)\ge -\lambda^*(0)(x_N+\sigma_-\rho) & \mbox{for}\quad x_N\le
-\sigma_-\rho.
\end{array}
$$
\item $|\nabla u|\le \lambda^*(0) (1+\tau)$ in $B_\rho$.
\end{enumerate}
If the origin is replaced by $x_0$ and the direction $e_N$ by the
unit vector $\nu$ we say that $u$ is of class
$F(\sigma_+,\sigma_-; \tau)$ in $B_\rho(x_0)$ in direction $\nu$.
\end{defi}

\begin{teo}\label{flat2} Let $p$ be Lipschitz continuous, $1<p_{min}\le p(x)\le p_{max}<\infty$, $\lambda$ H\"older continuous with $0<\lambda_1\le \lambda(x)\le \lambda_2<\infty$ and modulus of continuity $\omega_\lambda(r)=C_*r^{\alpha^*}$.  Then,
there exist $\sigma_0>0$ and $C_0>0$ such that if  $C_*\rho^{\alpha^*}\le \lambda^*(0)\sigma$ and  $0<\sigma<\sigma_0$,
$$
u\in F(\sigma,1; \sigma)\ in\ B_1 \mbox{ with power }p(x) \mbox{ and } |\nabla p|\le C\sigma \mbox{ in } B_\rho\ \ \mbox{implies}
$$
$$
u \in F(2\sigma,C_0\sigma;\sigma) \ in\ B_{\rho/2}.
$$
\end{teo}
\begin{proof} By rescaling, we may assume that $\rho=1$ and $osc_{B_1}\lambda^*\le C_*\rho^{\alpha^*}$.

Then, we proceed as in \cite{DP}, Lemmas 6.5, 6.6 and Theorem 6.3. One of the differences in our case is that $\lambda^*$ is not a constant. Moreover, we cannot assume that $\lambda^*(0)=1$. First, we
construct, for $\kappa>0$, a barrier $v$ as a solution to
$$
\begin{cases}
\Delta_{p(x)}v=0\qquad&\mbox{in }D\setminus B_r(\xi)\\
v=0\qquad&\mbox{on }\partial D\setminus B_1\\
v=\lambda^*(0)(1+\sigma)(\sigma-x_N)\qquad&\mbox{on }\partial D\cap B_1\\
v=-\lambda^*(0)(1-\kappa\sigma)x_N\qquad&\mbox{on }\partial B_r(\xi)
\end{cases}
$$

Here the set $D$ is constructed as in \cite{DP}. As in that paper, we want to prove that there exists $x_\xi\in\partial B_r(\xi)$ such that $v(x_\xi)\ge u(x_\xi)$ if $\kappa$ is large enough.

By contradiction, by Lemma \ref{distB}, if $v\le u$ on $\partial B_r(\xi)$ there holds that
$$
|\nabla v(z)|\ge \lambda^*(z)
$$
where $z\in \partial D\cap B_{1/2}\cap \partial\{u>0\}$.

Then,
$$
|\nabla v(z)|\ge \lambda^*(0)(1-\sigma).
$$

In order to get the contradiction we need the following estimate:
$$
|\nabla v(z)|\le \lambda^*(0)(1+C\sigma-c\kappa\sigma).
$$
For that purpose, we proceed again as in \cite{DP} by constructing a barrier for $v$ of the form $v_1-\kappa\sigma v_2$ where $v_1$ and $v_2$ are the same functions of \cite{DP}, Claim 6.8. One can check, as in \cite{DP}, that $v_1$ is a supersolution and $v_2$ is a subsolution to an elliptic equation in nondivergence form in such a way that $\Delta_{p(x)}(v_1-\kappa\sigma v_2)\le 0$. The difference in our case is that this equation has first order terms. But these terms are bounded by $L\sigma\log 2$ since by construction $\tfrac12\le|\nabla (v_1-\kappa\sigma v_2)|\le 2$.

In this way the results corresponding to Lemmas 6.5 and 6.6 in \cite{DP} are proved. In order to finish the proof of the theorem we proceed as in \cite{DP}, Theorem 6.3. We consider the function
$$
w(x)=\lambda^*(0)(1+\sigma)(\sigma-x_N)-u(x)\ge0\quad\mbox{in }B_{2r}(\xi)
$$
and prove that $w(x_\xi)\le C\sigma$ where $x_\xi\in \partial B_r(\xi)$ is
such $v(x_\xi)\ge u(x_\xi)$. Finally, in order to apply Harnack inequality to
get $w(x)\le C\sigma$ in $B_r(\xi)$ we observe that $w$ satisfies
$$
|\L w|\le C\sigma\qquad\mbox{in }B_{2r}(\xi)
$$
where $\L$ is the linear operator given in \eqref{nodiver2} such that $\L u=0$ (observe that at this stage we already know that $|\nabla u|\ge\lambda^*(0)/2$ in $B_{4r}(\xi)$.
\end{proof}

\begin{teo}\label{nodegegrad} Let $p$ be Lipschitz continuous, $1<p_{min}\le p(x)\le p_{max}<\infty$, $\lambda$ H\"older continuous with $0<\lambda_1\le \lambda(x)\le \lambda_2<\infty$ and modulus of continuity $\omega_\lambda(r)=C_*r^{\alpha^*}$.
For every $\delta>0$ there exist $\sigma_{\delta}>0$ and
$C_{\delta}>0$ such that if  $C_*\rho^{\alpha^*}\le \lambda^*(0)\sigma$, $0<\sigma<\sigma_{\delta}$,
$$
u\in F(\sigma,1; \sigma)\ in\ B_\rho \mbox{ with power } p(x) \mbox{ and } |\nabla p|\le C\sigma\ \ \mbox{implies}
$$
$$|\nabla u|\geq \lambda^*(0)(1-\delta)
\ in\ B_{\rho/2}\cap \{x_N\leq -C_{\delta} \sigma\}.
$$

\end{teo}

\begin{proof}
 The proof follows as  Theorem 6.4 in \cite{DP}.
\end{proof}

\medskip

\subsection{Nonhomogeneous blow-up}

\begin{lema}\label{9.1}
Let $u_k\in F(\sigma_k,\sigma_k;\tau_k) \in B_{\rho_k}$ with power $p_k(x)$ and coefficient $\lambda_k(x)$ such that $|\nabla p_k|\le L$, $1<p_{min}\le p_k(x)\le p_{max}<\infty$, $\lambda_k^*$  H\"older with exponent $\alpha^*$ and constant $C_*$, $0<\lambda_1\le \lambda_k(x)\le \lambda_2<\infty$. Assume
$\sigma_k \to 0$, $\tau_k \sigma_k^{-2}\to 0$ and $\rho_k^{\alpha^*}\leq
\rho_0\tau_k$ with $\rho_0>0$. For $y\in B_1'$, set
\begin{align*}&f_k^+(y)= \sup\{h: (\rho_k y, \sigma_k \rho_kh)\in
\partial\{u_k>0\}\},\\& f_k^-(y)= \inf\{h: (\rho_k y, \sigma_k
\rho_kh)\in \partial\{u_k>0\}\}.\end{align*} Then, for a
subsequence,
\begin{enumerate}
\item $f(y)=\limsup_{\stackrel{z\to y} {k\to \infty}}f_k^+(z)=
\liminf_{\stackrel{z\to y} {k\to \infty}}f_k^-(z)\ \ for\ all\
y\in B_1'.$

\smallskip

 Further, $f_k^+\to f$, $f_k^-\to f$ uniformly,
 $f(0)=0$, $|f|\leq 1$ and $f$ is continuous.

\medskip

 \item $f$ is
subharmonic.
\end{enumerate}
\end{lema}
\begin{proof}
(1) is the analogue of Lemma 5.3 in \cite{ACF}. The proof is based
on Theorem 6.3 and is identical to the one of Lemma 7.3 in
\cite{AC}.

For the reader's convenience, we write down the proof of (2) that is a little bit different from the one in \cite{AC} since  we do not have the homogeneity of the operator so that we need to keep track of the coefficient $\lambda_k^*(0)$. Also, our assumption in this and the ongoing sections is that $\lambda(x)$ is H\"{o}lder continuous as compared to the assumption in \cite{AC}.

We assume by taking $\widetilde{u}_k(x)=u_k(\rho_k x)/\rho_k$,  $\tilde p_k(x)=p_k(\rho_k x)$  and $\tilde\lambda_k(x)= \lambda_k(\rho_k x)$
that $u_k\in F(\sigma_k,\sigma_k;\tau_k)$ in $B_1$ with power $\tilde p_k$ and coefficient $\tilde\lambda_k$. We drop the tildes but recall that now $|\nabla p_k|\le L\rho_k$ and $|\lambda_k^*(x)-\lambda_k^*(0)|\le C^*\rho_k^{\alpha^*}|x|^{\alpha^*}$.

Observe that by the H\"{o}lder continuity of the original $\lambda_k^*$ we have that,
\begin{equation}\label{holder*}
\lambda_k^*(x)\geq
\lambda_k^*(0)-C^*\rho_k^{\alpha^*}=\lambda_k^*(0) (1-C\rho_k^{\alpha^*})
\end{equation}

Let us assume, by contradiction, that there is a ball
$B'_{\rho}(y_0)\subset B_{1}$ and a harmonic function $g$ in a
neighborhood of this ball, such that
$$
g>f  \mbox{ on } \partial B'_{\rho}(y_0)\quad  \mbox{ and } \quad
f(y_0)>g(y_0).
$$
Let,
$$Z^{+} =\{x\in B_1\,/\, x=(y,h),\ y\in B'_{\rho}(y_0),
h>\sigma_k g(y)\},$$ and similarly $Z_0$ and $Z^-$. As in Lemma
7.5 in \cite{AC}, using  the same test function and the
Representation Theorem \ref{repteo}  we
arrive at,
\begin{equation}\label{igualrep}\int_{\{u_k>0\}\cap Z_0} |\nabla u_k|^{p_k-2} \nabla
u_k \cdot \nu\, d\H= \int_{\partial_{red}\{u_k>0\}\cap Z^+} q_{u_k}(x) \,
d\H.\end{equation} As $u_k\in F(\sigma_k,\sigma_k,\tau_k)$ we have
that $|\nabla u_k|\leq \lambda_k^*(0)(1+\tau_k)$ and,  by Theorem
\ref{blow3}, there holds that $q_{u_k}(x)={\lambda_k^*(x)}^{p_k(x)-1}$
for $\H- a.e$ point in $\partial_{red}\{u_k>0\}$. Therefore,

\begin{equation}\label{igualrep2}\int_{\{u_k>0\}\cap Z_0} |\nabla u_k|^{p_k-2} \nabla
u_k \cdot \nu\, d\H= \int_{\partial_{red}\{u_k>0\}\cap Z^+}
{\lambda_k^*}^{p_k-1} \, d\H.\end{equation}

Applying the estimate \eqref{holder*} to \eqref{igualrep2} and,  assuming for simplicity that $\lambda_k^*(0)\geq 1$ we have,
\begin{align*}
{\lambda_k^*(0)}^{p_k^--1}\,(1-C\rho_k^{\alpha^*})^{p_k^+-1}\,
\H&(\partial_{red}\{u_k>0\}\cap Z^+) \\
&\leq\int_{\{u_k>0\}\cap Z_0}
|\nabla u_k|^{p_k-2} \nabla u_k \cdot \nu\, d\H\\
& \leq {\lambda_k^*(0)}^{p_k^+-1}
(1+\tau_k)^{p_k^+-1}
\H(\{u_k>0\}\cap Z_0)
\end{align*}

Then we have,
\begin{equation}\label{desilambda1}
\H(\partial_{red}\{u_k>0\}\cap Z^+)\leq {\lambda_k^*(0)}^{p_k^+-p_k-}
\Big(\frac{1+\tau_k}{1-C^*\rho_k^{\alpha^*}}\Big)^{p_k^+-1}
\H(\{u_k>0\}\cap Z_0).
\end{equation}

On the other hand, by the excess area estimate in Lemma 7.5 in \cite{AC} we have that,
$$
\H(\partial_{red}E_k\cap Z)\geq \H(Z_0)+c\sigma_k^2,
$$
where $Z=B'_{\rho}(y_0)\times \mathbb{R}$ and $E_k=\{u_k>0\}\cup Z^-$.

We also have,
$$\H(\partial_{red}E_k\cap
Z)\leq \H(Z^+\cap\partial_{red}\{u_k>0\})+\H(Z_0\cap\{u_k=0\}).$$
Using these two inequalities and the fact that
$\H(Z_0\cap\partial\{u_k>0\})=0$ (if this is not true we replace
$g$ by $g+c_0$ for a small constant $c_0$) we have that,
\begin{equation}\label{exces}\H(\partial_{red}\{u_k>0\}\cap Z^+)\geq \H(Z_0\cap
\{u_k>0\} )+c\sigma_k^2.\end{equation} Finally by
\eqref{desilambda1} and \eqref{exces} we have that,
\begin{align*}
\H(\{u_k>0\}\cap Z_0)+c\sigma_k^2\leq&
{\lambda_k^*(0)}^{p_k^+-p_k-}
\Big(\frac{1+\tau_k}{1-C^*\rho_k^{\alpha^*}}\Big)^{p_k^+-1} \H(\{u_k>0\}\cap
Z_0).
\end{align*}
Therefore, since $p_k^+-p_k^-\le L\rho_k$,
$$
c\sigma_k^2\leq \Big[{\lambda^*_k(0)}^{p_k^+-p_k-}
\Big(\frac{1+\tau_k}{1-C^*\rho_k^{\alpha^*}}\Big)^{p_k^+-1}-1\Big]
\H(\{u_k>0\}\cap Z_0)\leq C
(\tau_k+\rho_k^{\alpha^*}).
$$

Observe that if $\lambda_k^*(0)\leq 1$, we arrive at the same estimate.

Finally, since $\rho^{\alpha^*}\leq
\rho_0 \tau_k$  this contradicts the fact that
$\frac{\tau_k}{\sigma_k^2}\to 0$ as $k\to \infty$.
\end{proof}

\begin{lema}
There exists a positive constant $C=C(N)$ such that, for any $y\in
B'_{r/2}$,
$$\int_0^{1/4} \frac{1}{r^2} \Big(\pint_{\partial
B'_r(y)}f-f(y)\Big)\leq C_1.$$

\end{lema}
\begin{proof}
It follows as  Lemma 8.3 in \cite{DP},  by Remark \ref{T} and
Theorem \ref{nodegegrad}.

The only difference is that the functions $w_k = (u_k(y,h)+h)/\sigma_k$ verify a second order elliptic equation in non--divergence form with H\"older principal coefficients and bounded first order coefficients.

As in the proof of Lemma \ref{9.1}, since $|\nabla p_k|\to 0$, these first order coefficients converge to 0 and so, by the $W^{2,q}$ regularity estimates of \cite{GT}, Chapter 9, we can pass to the limit to discover that $w_k\to w$ and $w$ satisfies a second order elliptic equation in non--divergence form with constant coefficients with only principal part.

After that, the proof follows without any change as that of Lemma 8.3 in \cite{DP}.
\end{proof}

With these two lemmas we have by Lemma 7.7 and Lemma 7.8 in
\cite{AC},

\begin{lema}\label{flat}\begin{enumerate}

\item $f$ is Lipschitz in $\bar{B}'_{1/4}$ with Lipschitz constant
depending on $C_1$ and $N$. \item There exists a constant
$C=C(N)>0$ and for $0<\theta<1$, there exists
$c_{\theta}=c(\theta,N)>0$, such that we can find a ball $B_r'$
and a vector $l\in \mathbb{R}^{N-1}$
 with
$$c_{\theta} \leq r\leq \theta,\ \ |l|\leq C, \ \ \mbox{ and }
f(y)\leq l .y+\frac{\theta}{2} r \ \ \mbox{ for } |y|\leq r.$$

 \end{enumerate}
\end{lema}
And, as in Lemma 7.9 in \cite{AC} we have,
\begin{lema}\label{flat31}
Let $\theta$, $C$, $c_{\theta}$ as in Lemma \ref{flat}.  There
exists a positive constants $\sigma_{\theta}$,  such that
\begin{equation}\label{flat51}u\in F(\sigma,\sigma;\tau) \mbox{ in } B_{\rho} \mbox{ in
direction } \nu\end{equation} with $\sigma\leq \sigma_{\theta}, \
\tau\leq \sigma_{\theta} \sigma^2$ and $\rho^{\alpha^*}\leq \rho_0 \tau$, implies
$$u\in F(\theta\sigma,1; \tau) \mbox{ in } B_{\bar{\rho}} \mbox{ in direction }
\bar{\nu}$$ for some $\bar{\rho}$ and $\bar{\nu}$ with
$c_{\theta}\rho\leq \bar{\rho}\leq \theta \rho$ and
$|\bar{\nu}-\nu|\leq C\sigma$, where
$\sigma_{\theta}=\sigma_{\theta}(\theta,N)$.
\end{lema}

\begin{lema}\label{flat3}
Given $0<\theta<1$, there exist positive constants
$\sigma_{\theta}$,  $c_{\theta}$ and $C$ such that
\begin{equation}\label{flat5}u\in F(\sigma,1;\tau) \mbox{ in } B_{\rho} \mbox{ in
direction } \nu\end{equation} with $\sigma\leq \sigma_{\theta},
\tau\leq \sigma_{\theta} \sigma^2$ and $\rho^{\alpha^*}\leq \rho_0 \tau$, then
$$u\in F(\theta\sigma,\theta\sigma;\theta^2 \tau) \mbox{ in } B_{\bar{\rho}} \mbox{ in direction }
\bar{\nu}$$ for some $\bar{\rho}$ and $\bar{\nu}$ with
$c_{\theta}\rho\leq \bar{\rho}\leq \frac{1}{4} \rho$ and
$|\bar{\nu}-\nu|\leq C\sigma$, where
$c_{\theta}=c_{\theta}(\theta,N)$,
$\sigma_{\theta}=\sigma_{\theta}(\theta,N)$.
\end{lema}
\begin{proof}
We obtain the improvement of the value $\tau$ inductively. If
$\sigma_{\theta}$ is small enough, we can apply Theorem \ref{flat2}
and obtain
$$u\in F(C\sigma,C\sigma;\tau) \mbox{ in } B_{\rho/2} \mbox{ in direction } \nu.$$ Then
for $0<\theta_1\leq \frac{1}{2}$ we can apply Lemma \ref{flat31},
if again $\sigma_{\theta}$ is small, and we obtain
\begin{equation}\label{theta1}u\in F(C\theta_1\sigma,C\sigma;\tau) \mbox{ in }
B_{r_1} \mbox{ in direction } \nu_1\end{equation} for some
$r_1, \nu_1$ with
$$c_{\theta_1}\rho\leq 2r_1\leq \theta_1\rho, \mbox{ and }
|\nu_1-\nu|\leq C\sigma.
$$

In order to improve $\tau$, we consider the functions
$U_{\ep}=\big(|\nabla u| - \lambda^*(0) - \ep\big)^+$ and
$U_0=\big(|\nabla u| - \lambda^*(0)\big)^+$ in $B_{2r_1}$. By
Lemma \ref{const},
we know that $U_{\ep}$ vanishes in a neighborhood of the free
boundary. Since $U_{\ep}>0$ implies $|\nabla u|>\lambda^*(0)+\ep$, the
closure of $\{U_{\ep}>0\}$ is contained in $\{|\nabla
u|>\lambda^*(0)+\ep/2\}$.

Since $|\nabla u|$ is bounded from above in $B_{2r_1}$, and from
below in the set $\{|\nabla u|>\lambda^*(0)+\ep/2\}$ the hypotheses
of Lemma \ref{ecugradu} are
satisfied, and we have that $v=|\nabla u|$
satisfies,
$$
- \mbox{div} D\nabla v+B\nabla v\leq\mbox{div} H\quad \mbox{ in }
\{|\nabla u|>\lambda^*(0)+\ep/2\}
$$

 Hence $U_{\ep}$ satisfies
$$
- \mbox{div} D\nabla U_{\ep}+B\nabla U_{\ep}\leq\mbox{div} H\quad
\mbox{ in } \{U_\ep>0\}
$$

Extending the operator  by a uniformly elliptic
operator with principal part in divergence form with ellipticity constant $\beta$ and $H$ by $\tilde H$ with $\|\tilde H\|_\infty\le C\|H\|_\infty$ we get,

$$\begin{cases}- \mbox{div} \widetilde{D}\nabla U_{\ep}+\widetilde{B}\nabla U_{\ep}\leq\mbox{div}
\widetilde{H}\quad &\mbox{ in } B_{2r_1}
\\
U_{\ep}\leq \lambda^*(0)\tau &\mbox{ on }\partial
B_{2r_1}\end{cases}$$

 and $r_1\leq \theta_1\rho\leq \rho/4$.
 Then, $U_{\ep}\leq
\lambda^*(0)\tau+C(N,L,\beta)r_1\|H\|_{\infty}$. Let
$S=\lambda^*(0)\tau+Cr_1\|H\|_{\infty}$.

Let $W=S-U_{\ep}\geq 0$. Then
$$- \mbox{div} \widetilde{D}\nabla W+\widetilde{B}\nabla W\geq\mbox{div}
\widetilde{H}\quad \mbox{ in } B_{2r_1}.$$

By the weak Harnack inequality
(see \cite{GT} Theorem 8.18) we have that, if $1<q<\frac{N-2}{2}$
then,
$$\frac{1}{r_1^{N/q}} \|W\|_{L^q(B_{2r_1})}\leq C_1[\inf_{B_{r_1}} W+
\|H\|_{L^{\infty}(B_{2r_1})}r_1].$$ In $\{u=0\}$,
$W=S$. Moreover $u=0$ in $B_{r_1/4}(\frac{3}{4}
r_1\bar{\nu})$ since $\theta_1\leq 1/4$. Therefore,
$$S\leq C_2[S-\sup_{B_{r_1}} U_{\ep}+
\|H\|_{L^{\infty}(B_{2r_1})}r_1].
$$
Then,
$$\sup_{B_{r_1}} U_{\ep}\leq \Big(1-\frac{1}{C_2}\Big) S+C_3r_1\leq
\delta  \lambda^*(0)\tau+ C_4r_1$$ with $0<\delta<1$. And we
have,
$$\sup_{B_{r_1}}|\nabla u|\leq \lambda^*(0) (1+\delta
\tau)+C_4r_1$$ with $C_4=C_4(N,\lambda^*(0),L,\beta)$.

Since $r_1\leq \theta_1 \rho$ and $\rho\leq \rho^{\alpha^*}\leq \rho_0 \tau$, we have

$$\sup_{B_{r_1}}|\nabla u|\leq \lambda^*(0) (1+\delta
\tau)+C_4r_1\leq \lambda^*(0) (1+\delta \tau+ C\theta_1 \rho
)\leq\lambda^*(0)(1+\tau(\delta+C\theta_1\rho_0)).$$

Let us choose $\theta_1$
such that $C\rho_0\theta_1+\delta<1$. Take
$\theta_0=\max\{\theta_1^{\alpha^*/2},(\delta+C\rho_0\theta_1)^{1/2}\}$.

 We have
$$u\in F(\theta_0\sigma,1;\theta_0^2\tau) \mbox{ in } B_{r_1} \mbox{ in
direction } \nu_1.$$

Moreover, $r_1^{\alpha^*}\leq  \theta_1^{\alpha^*} \rho^{\alpha^*}\leq  \theta_0^2 \rho^{\alpha^*}
\leq \rho_0 \theta_0^2 \tau$. We also have,
 $\theta_0\sigma \leq \theta_0\sigma_{\theta_1}\leq \sigma_{\theta_1}$ and
$\theta_0^2 \tau \leq \theta_0^2 \sigma_{\theta_1}\sigma^2=\sigma_{\theta_1}(\theta_0 \sigma)^2 $.

Then, we can repeat this argument a finite number
of times, and we obtain
$$u\in F(\theta_0^m\sigma,1;\theta_0^{2m}\tau) \mbox{ in } B_{r_1...r_m} \mbox{ in
direction } \nu_m,$$ with
$$c_{\theta_j}\leq 2 r_j\leq \theta_j, \mbox{ and }
|\nu_m-\nu|\leq \frac{C}{1-\theta}\sigma.$$ Finally we choose $m$
large enough such that $\theta_0^m\leq \theta$, we have
that
$$u\in F(\theta\sigma,1;\theta^{2}\tau) \mbox{ in } B_{r_1...r_m} \mbox{ in
direction } \nu_m,$$
and
using Theorem \ref{flat2} we have if $\sigma\leq
\sigma_{\theta_1,\rho_0}$, $\tau\leq
\sigma_{\theta_1,\rho_0}\sigma^2$ and $\rho^{\alpha^*}\leq \rho_0\tau$ the
desired result.
\end{proof}

\subsection{Smoothness of the free boundary}

\begin{teo}\label{regfin}
Suppose that $u$ is a minimizer of $\J$ in $\K$ and $D\subset\subset\Omega$.
Assume $p$ is Lipschitz and $\lambda$ is H\"older.
Then, there exist positive constants $\bar{\sigma}_0$, $C$ and
$\gamma$ such that if
$$
u\in F(\sigma,1;\infty) \quad in\ B_{\rho}(x_0)\subset D \mbox{
in direction } \nu
$$
with $\sigma\leq \bar{\sigma}_0$, $\rho\leq
\bar{\rho}_0(\bar{\sigma}_0,\sigma)$, then
$$
B_{\rho/4}(x_0)\cap \partial\{u>0\} \mbox{ is a } C^{1,\gamma}
\mbox{ surface.}
$$

More precisely, a graph in direction $\nu$ of a
$C^{1,\gamma}$ function and, for any $x_1$, $x_2$ on this surface
$$
|\nu(x_1)-\nu(x_2)|\leq C\sigma
\Big|\frac{x_1-x_2}{\rho}\Big|^{\gamma}
$$
\end{teo}
\begin{proof}
See Theorem 9.3 in \cite{MW1}.
\end{proof}

\begin{remark}\label{remfin} By the nondegeneracy (Corollary \ref{prom12}) and by \eqref{asymp},
 we have that
 for $x_0\in \partial_{red}\{u>0\}$
we have that
 $u\in F(\sigma_{\rho},1;\infty)$ in $B_{\rho}(x_0)$ in direction
$\nu_u(x_0)$, with $\sigma_{\rho}\to 0$ as $\rho\to 0$. Hence, applying Theorem \ref{regfin} we
have,\end{remark}

\begin{teo}\label{teo.regularity}
Let $u$ be a local minimizer of $\J$ in $\K$ with power $p\in Lip$ and coefficient $\lambda\in C^\alpha$. Then, for any $x_0\in \partial_{red}\{u>0\}$
there exist $r>0$ and $0<\gamma<1$ such that $B_r(x_0)\cap\partial\{u>0\}$ is a
$C^{1,\gamma}$ surface. Thus, for every $D\subset\subset\Omega$ there exists $0<\gamma<1$ such that $D\cap\partial_{red}\{u>0\}$ is a $C^{1,\gamma}$ surface and moreover, $\H(\partial\{u>0\}\setminus \partial_{red}\{u>0\}=0$.
\end{teo}

\appendix

\section{The spaces $L^{p(\cdot)}(\Omega)$ and $W^{1,p(\cdot)}(\Omega)$} \label{appA1}

\setcounter{equation}{0}

Let $p :\Omega \to  [1,\infty)$ be a measurable bounded function, called a variable exponent on $\Omega$ and
denote $p_{max} = ess sup \,p(x)$ and $p_{min} = ess inf \,p(x)$. We define the variable exponent Lebesgue
space $L^{p(\cdot)}(\Omega)$ to consist of all measurable functions $u :\Omega \to \R$
for which the modular
$\varrho_{p(\cdot)}(u) = \int_{\Omega} |u(x)|^{p(x)}\, dx$ is finite. We define the Luxemburg norm on this space by
$$
\|u\|_{L^{p(\cdot)}(\Omega)} = \|u\|_{p(\cdot)}  = \inf\{\lambda > 0: \varrho_{p(\cdot)}(u/\lambda)\leq 1 \}.
$$

This norm makes $L^{p(\cdot)}(\Omega)$ a Banach space.

One central property of these spaces (since $p$ is bounded) is that $\varrho_{p(\cdot)}(u_i)\to 0$
 if and only $\|u_i\|_{p(\cdot)}\to 0$, so that the norm and modular topologies coincide.

\begin{remark}\label{equi}
Observe that we have the following estimate,
$$\|u\|_{L^{p(\cdot)}(\Omega)} \leq  \max\Big\{\Big(\int_{\Omega} |u|^{p(x)})\, dx\Big)
^{1/{p_{min}}}, \Big(\int_{\Omega} |u|^{p(x)}\, dx\Big)
^{1/{p_{max}}}\Big\}$$

In fact. If $\di\int_{\Omega} |u|^{p(x)}\, dx= 0$ then $u=0$ a.e and
the result follows. If $\int_{\Omega} |u|^{p(x)}\, dx\neq 0$, take
$k=\max\Big\{\Big(\int_{\Omega} |u|^{p(x)})\, dx\Big)
^{1/{p_{min}}}, \Big(\int_{\Omega} |u|^{p(x)}\, dx\Big)
^{1/{p_{max}}}\Big\}$. Then we have,
$$\int_{\Omega} \Big(\frac{|u|}{k}\Big)^{p(x)}\, dx\leq  \max\Big\{\frac{1}{k^{p_{min}}},
\frac{1}{k^{p_{max}}} \Big\}\int_{\Omega} |u|^{p(x)}\, dx\leq 1$$
therefore $\|u\|_{L^{p(\cdot)}(\Omega)} \leq k$ and the result
follows.
\end{remark}

Let $W^{1,p(\cdot)}(\Omega)$ denote the space of measurable functions $u$ such that $u$ and the distributional derivative $\nabla u$ are in $L^{p(\cdot)}(\Omega)$. The norm

$$
\|u\|_{1,p(\cdot)}:= \|u\|_{p(\cdot)} + \| |\nabla u| \|_{p(\cdot)}
$$
makes $W^{1,p(\cdot)}$ a Banach space.

\begin{teo}\label{ref}
Let $p'(x)$ such that, $$\frac{1}{p(x)}+\frac{1}{p'(x)}=1.$$ Then
$L^{p'(\cdot)}(\Omega)$ is the dual of $L^{p(\cdot)}(\Omega)$.
Moreover, if $p_{min}>1$, $L^{p(\cdot)}(\Omega)$ and $W^{1,p(\cdot)}(\Omega)$ are
reflexive.
\end{teo}

\begin{teo}\label{imb}
Let $q(x)\leq p(x)$, then
 $L^{p(\cdot)}(\Omega)\hookrightarrow L^{q(\cdot)}(\Omega)$
continuously.
\end{teo}

We define the space $W_0^{1,p(\cdot)}(\Omega)$ as the closure of the $C_0^{\infty}(\Omega)$ in $W^{1,p(\cdot)}(\Omega)$. Then we have the following version of Poincare's inequity,

\begin{lema}\label{poinc} If $p(x)$ is continuous in $\overline{\Omega}$,  there exists a constant $C$ such that for every $u\in W_0^{1,p(\cdot)}(\Omega)$,
$$
\|u\|_{L^{p(\cdot)}(\Omega)}\leq C\|\nabla u\|_{L^{p(\cdot)}(\Omega)}
$$
\end{lema}

For the proof of these results, and more about these spaces see \cite{KR}.

In order to have better properties of these spaces, we need more hypotheses on the regularity of $p(x)$.

We say that $p$ is log-H\"{o}lder continuous
if there exists a constant $C$ such that
$$|p(x) - p(y)| \leq \frac{C}{\big|\log\, |x - y|\big|}$$
if $|x - y| < 1/2 $.

It was proved in  \cite{Sam}, Theorem 3  that if one assumes that $p$ is log-H\"{o}lder continuous then,  $C^{\infty}$ is dense in $W^{1,p(\cdot)}(\Omega)$ (see also \cite{BS} and \cite{Di}). See \cite{DHN} for more references on this topic.

\section{Results on $p(x)-$harmonic and subharmonic functions}\label{appA}

\setcounter{equation}{0}
\renewcommand{\theequation}{B.\arabic{equation}}

In this section we will give some of the properties of $p(x)-$harmonic and subharmonic functions.
Some of them are known results and  others are new. For the reader's convenience we will list all the results,
and give the reference when it corresponds. Here $\omega(r)$ is the modulus of continuity of $p(x)$. We will state  which is the type of $\omega$ that
we are considering for each result.

\begin{remark}\label{desip} For any $x$ fixed we have the following inequalities
\begin{align*}
&|\eta-\xi|^{p(x)}\leq C (|\eta|^{p(x)-2} \eta-|\xi|^{p(x)-2} \xi)
(\eta-\xi)&\quad  \mbox{ if } p(x)\geq 2,\\
&
 |\eta-\xi|^2\Big(|\eta|+|\xi|\Big)^{p(x)-2}
\leq C (|\eta|^{p(x)-2} \eta-|\xi|^{p(x)-2} \xi)
(\eta-\xi)&\quad  \mbox{ if } p(x)< 2,\\
\end{align*}
These inequalities say that the function $A(x,q)=|q|^{p(x)-2}q$ is strictly monotone. Then, the comparison principle holds since it follows from the monotonicity of $A(x,q)$.

\end{remark}

\bigskip

The following result, a Cacciopoli type inequality, is included in the proof of Lemma 6 in
\cite{HK},

\begin{lema}\label{caccio} Assume $p(x)$ is bounded and let $u\in W^{1,p(\cdot)}(\Omega)$ be a nonnegative subsolution of the
problem
\begin{equation}
\Delta_{p(x)} u=0 \mbox{
in }\Omega.
\end{equation}
Then, for any $B_r\subset \Omega$
$$\int_{B_{r/2}} |\nabla u|^{p(x)}\, dx \leq
C\int_{B_{r}}\left(\frac{u}{r}\right)^{p(x)}\, dx,$$ where
$C=C(p_{min},p_{max})$.

\end{lema}

\begin{proof}
See inequality (5) in the proof of Lemma 6 of \cite{HK}.
\end{proof}

\begin{lema}\label{har} Assume $\omega(r)=C (log \frac{1}{r})^{-1}$ and let $u$ be a nonnegative solution of the
problem
\begin{equation}
\Delta_{p(x)} u=0 \mbox{
in }\Omega.
\end{equation}
Then, there exists a constant $C$ such that
$$\sup_{{B_r}(x_0)}u\leq C(\inf_{{B_r}(x_0)}u+r)$$ for any
$r $ with $B_{10r}(x_0)\subset \Omega$. The constant
depends on $N,\omega(.), p_{min}$ and the $L^1(B_r)$-norm of $|u|^{p(\cdot)}$.

\end{lema}
\begin{proof}
See Corollary 5.13 in \cite{HKLMP}.
\end{proof}

\begin{remark}\label{harna2}
Assume $u$ is a nonnegative solution of the
problem
\begin{equation}
\Delta_{p(x)} u=0 \mbox{
in }\Omega.
\end{equation}
Let $R, \bar{y}$ such that $B_{10R}(\bar{y})\subset \Omega$, $x_0\in \Omega$ and $r>0$. Let $\bar{x}=\frac{\bar{y}-x_0}{r}$  and
$\bar{u}(x)=\frac{u(x_0+rx)}{r}$. Then, for any $\rho<R/r$ we have,
$$\sup_{{B_{\rho}}(\bar{x})}\bar u\leq C(\inf_{{B_\rho}(\bar{x})}\bar{u}+\rho),$$
where $C$ is the constant of the previous Lemma. In particular, $C$ may be taken independent of $r$ (it depends on the $L^1(\Omega)-$norm of $|u|^{p(\cdot)}$).

\end{remark}
\begin{proof}
Let $|x-\bar{x}|<\rho$ and $y=x_0+rx$ then $|y-\bar{y}|=r|x-\bar{x}|<r\rho<R$. Since,
$$\sup_{B_{\rho r}(\bar{y})}u\leq C (\inf_{{B_{\rho r}}(\bar{y})}{u}+\rho r)$$
there holds that
$$\sup_{B_{\rho r}(\bar{y})}\frac{u}{r}\leq C (\inf_{{B_{\rho r}}(\bar{y})}\frac{u}{r}+\rho ).
$$
Then,
$$\sup_{B_{\rho }(\bar{x})}\bar{u}\leq C (\inf_{{B_{\rho }}(\bar{x})}\bar{u}+\rho ).
$$
\end{proof}

The following result was proved in Theorem 2.2 in \cite{AM},

\begin{teo}\label{regu} Assume $\omega(r)=C_0 r^{\alpha_0}$  for some $0<\alpha< 1$, and let $u$ be a solution of the
problem
\begin{equation}
\Delta_{p(x)} u=0 \mbox{
in }\Omega.
\end{equation}
Then, for any $\Omega'\subset \subset\Omega$ there exists a constant $C$ depending on $\|| u|^{p(x)}\|_{L^{1}(\Omega)}$, $\||\nabla u|^{p(x)}\|_{L^{1}(\Omega)}$, $p_{min},p_{max}$,  $\omega(r)$ and $\Omega'$
such that
$$\|u\|_{C^{1,\alpha}(\Omega')}\leq C.$$

\end{teo}

The following is a consequence of the $C^{1,\alpha}$ regularity of
the solutions and the Cacciopoli inequality
\begin{lema}\label{gradient-estimate}
Assume $\omega(r)=C_0 r^{\alpha_0}$. Let $u$ be a solution  of the problem
\begin{equation}
\Delta_{p(x)} u=0 \mbox{
in }B_R(y).
\end{equation}
Then, there exists a constant $C$ depending on $\| |u|^{p(x)}\|_{L^{1}(B_R(y))}$, $\||\nabla u|^{p(x)}\|_{L^{1}(B_R(y))}$, $p_{min},p_{max}$ and  $\omega(r)$  such that if $R\leq 1$ we
have,
$$|\nabla u(y)|\leq C\Big(1+\frac{1}{R}\sup_{B_R(y)}
u\Big)^{p_+/p_-}$$
where $p_+=\sup_{B_R(y)}p$, $p_-=\inf_{B_R(y)}p$.
\end{lema}
\begin{proof}
By  Theorem \ref{regu}, we have for $x\in
B_{R/2}(y)$,
$$|\nabla u(x)-\nabla u(y)|\leq C|x-y|^{\alpha},$$
for some constants $C>0$ and $0<\alpha<1$. Therefore, if $x\in
B_{R/2}(y)$
$$|\nabla u(y)|\leq |\nabla u(x)| + CR^{\alpha}.$$
If $|\nabla u(y)|\geq 1$, $p_-=p_-(B_R(y))$,$p_+=p_+(B_R(y))$, recalling that $R\le1$ we get,
$$|\nabla u(y)|^{p_-}\leq |\nabla u(y)|^{p(x)} \leq C|\nabla u(x)|^{p(x)}+ C.$$
Integrating for $x\in B_{R/2}(y)$,
$$|\nabla u(y)|^{p_-}\leq  C\Big(1+\pint_{B_{R/2}(y)}|\nabla u(x)|^{p(x)}\Big).$$
Applying Cacciopoli inequality we have, since $R\leq 1$,
\begin{align*}|\nabla u(y)|^{p_-}&\leq  C_1\Big(1+\pint_{B_{R}(y)}
\Big(\frac{|u(x)|}{R}\Big)^{p(x)}\Big)\\&\leq
C_1\Big(2+\pint_{B_{R}(y)} \Big(\frac{|u(x)|}{R}\Big)^{p_+}\Big)\\& \leq C\Big(1+
\Big(\frac{1}{R}\sup_{B_{R}(y)}u(x)\Big)^{p_+}\Big).
\end{align*}
We obtain the desired result.
\end{proof}

\begin{remark}\label{nodiver}
In some of the proofs we  need to look at the $p(x)-$Laplacian as an operator in  non-divergence form. In those cases we have to assume $p(x)$ Lipschitz so that we can differentiate the function,
$$A(x,q)=|q|^{p(x)-2} q.$$
If  we take a function $u$, with $c_1\leq |\nabla u|\leq c_2$ differentiating we obtain,
\begin{align*}
a_{ij}(x,\nabla u)&=\frac{\partial{A_i}}{\partial q_j}(x,\nabla u)=|\nabla u|^{p-2}
\big(\delta_{ij}+\frac{(p-2)}{|\nabla u|^2} u_{x_i}
u_{x_j}\big)\\
\frac{\partial A_i}{\partial x_k}(x,\nabla u)& = |\nabla u|^{p-2}\log|\nabla u|\, p_{x_k} u_{x_i}
.\end{align*}
Then, we have a the following non-divergence form for the $p(x)-$Laplacian,
\begin{equation*}
\Delta_{p(x)} u= \L u
\end{equation*}
where
\begin{equation}\label{nodiver2}
\L w:=a_{ij}(x,\nabla u) w_{x_i x_j}+ |\nabla u|^{p-2}\log|\nabla u| \  p_{x_i} w_{x_i}.
\end{equation}

Observe that $a_{ij}=|\nabla u|^{p(x)-2}b_{ij}$ and $b_{ij}$ is uniformly elliptic with constant of ellipticity $\beta$,  independent of the gradient of $u$. We call
\begin{equation}\label{Tu}
Tw:= b_{ij}(x,\nabla u)w_{x_ix_j}+\log|\nabla u|p_{x_i}w_{x_i}.
\end{equation}
\end{remark}

The following Lemma is the construction of barriers required in several proofs.

\begin{lema}\label{exp}
Suppose that $p(x)$ is Lipschitz continuous. Let
$w_{\mu}=Me^{-\mu|x|^2}$, for $M>0$ and $r_1\ge|x|\ge r_2>0$. Then, there
exist ${\mu}_0,\ep_0>0$ such that, if $\mu>{\mu}_0$ and $\|\nabla
p\|_{\infty}\leq \ep_0$,
\begin{align*}&\mu^{-1}e^{\mu|x|^2}M^{-1}|\nabla w|^{2-p}\Delta_{p(x)}w_{\mu}\geq C_1 (\mu-C_2 \|\nabla
p\|_{\infty} |\log M|)
 \mbox{ in }B_{r_1}\setminus B_{r_2}.\end{align*}
Here $C_1,C_2$ depend only on $r_2,r_1, p_+,p_-$,\ \ \
${\mu}_0={\mu}_0(p_+,p_-,N,\|\nabla p\|_{\infty},r_2,r_1)$
and \newline
$\ep_0=\ep_0(p_+,p_-,r_1,r_2)$.
\end{lema}
\begin{proof}
First  note that by Remark \ref{nodiver} $$\Delta_{p(x)}w=|\nabla w|^{p-2}
\Big\{\frac{(p-2)}{|\nabla w|^2}\sum_{i,j}
w_{x_i}w_{x_j}w_{x_ix_j}+\triangle w+ \langle \nabla w,\nabla
p\rangle \log |\nabla w|\Big\}.$$ Computing, we have
\begin{align}\label{gradw}
w_{x_i}=-2\mu M x_i, e^{-\mu|x|^2},\  w_{x_i x_j}=M(4\mu^2 x_i
x_j-2\mu \delta_{ij}) e^{-\mu|x|^2},\  |\nabla w|=2M\mu
|x|e^{-\mu|x|^2}.
\end{align}
Therefore using \eqref{gradw}
 we obtain,
\begin{align*}\di &e^{\mu|x|^2}(2M\mu)^{-1}|\nabla w|^{2-p}\Delta_{p(x)}w\\
&=(p-2)(2\mu|x|^2-1)+(2\mu|x|^2-N)
 -  \langle x,\nabla p \rangle
(\log(M)+\log(|x|2)+\log\mu)+\mu\langle x,\nabla p\rangle |x|^2\\
&=(p-1)2\mu|x|^2+\mu \langle x,\nabla
p\rangle|x|^2 -(p-2+N) - \langle x,\nabla p \rangle (\log M+\log
\mu+\log (2 |x|))\\
&\ge (2(p_--1)r_2^2-r_1^3\|\nabla p\|_{\infty}) \mu-
r_1 \|\nabla p\|_{\infty}|\log \mu|-(p_+-2+N)- r_1 \|\nabla
p\|_{\infty} (\log M+ C_{r_1,r_2}))\\
& \ge (2 (p_--1)r_2^2- r_1
\|\nabla p\|_{\infty}(r_1^2+1 ) \mu
-(p_+-2+N)-r_1 \|\nabla p\|_{\infty}
(\log M+C_{r_1,r_2}).
\end{align*}

In the last inequality we have used that $\frac{\log{\mu}}{\mu}\leq 1$ if $\mu\geq 1$.

Let $\ep_0>0$ such that $$2 (p_--1) r_2^2- r_1 (r_1^2+1)\ep_0\geq
\frac{3}{2} (p_--1) r_2^2.$$ If $\|\nabla p\|_{\infty}\leq \ep_0$
we obtain

\begin{align*}\di e^{\mu|x|^2}(2M\mu)^{-1}&|\nabla w|^{2-p}\Delta_{p(x)}w\ge
\\&  \frac32 (p_--1)r_2^2\mu -(p_+-2+N)-r_1 \|\nabla
p\|_{\infty} (\log M+ C_{r_1,r_2}).
\end{align*}
Now, if we take $\mu\geq \mu_0 =\mu_0(p_+,p_-,N,r_2,r_1,\|\nabla
p\|_{L^{\infty}})$ we obtain that
\begin{align*}\di e^{\mu|x|^2}(2M\mu)^{-1}|\nabla w|^{2-p}&\Delta_{p(x)}w
 \geq C_1 (\mu-C_2 \|\nabla p\|_{\infty} |\log M|)).
\end{align*}
with $C_1,C_2$ depending only on $p_-,r_1,r_2$.

\end{proof}

\begin{lema}\label{ecugradu}
Assume $p(x)$ is Lipschitz. Let $u$ be a solution of the problem
\begin{equation}
{\Delta_{p(x)}}u=0 \mbox{ in
}\Omega.
\end{equation}
with $0<c_1\leq |\nabla u|\leq c_2$.
Then $v=|\nabla u|$ satisfies,
$$
- \mbox{div } D\nabla v+B\nabla v\leq\mbox{div } H\quad \mbox{ in }
\Omega
$$
where,
\begin{align*}
D_{ij}(x,\nabla u)&=|\nabla u|^{p-1}
\big(\delta_{ij}+\frac{(p-2)}{|\nabla u|^2} u_{x_i}
u_{x_j}\big),\\
H(x,\nabla u)&=|\nabla u|^{p-2}\log|\nabla u|\langle \nabla u,
\nabla p \rangle \nabla u,\\
B(x,\nabla u)&=|\nabla u|^{p-1}\log|\nabla u|\
\nabla p.
\end{align*}

\end{lema}
\begin{proof}
Let $\eta\in C_0^{\infty}(\Omega)$. Then, for each $k$ we have after  integration by parts,
$$0=\int_{\Omega} A(x,\nabla u) \nabla \eta_{x_k}\, dx=-
\int_{\Omega} \frac{\partial A}{\partial x_k }(x,\nabla u) \nabla \eta\, dx-
\int_{\Omega} a_{ij}(x,\nabla u) u_{x_j x_k} \eta_{x_i}\, dx.
$$
Observe that, by approximation, we get that the right hand side vanishes for $\eta\in W^{1,p(\cdot)}(\Omega)$.

Taking $\eta=u_{x_k} \psi$ with $\psi\in C_0^{\infty}(\Omega)$ we have, by using the ellipticity of $a_{ij}$ (see Remark \ref{nodiver}),
\begin{align*}
0=&-\int_{\Omega} \frac{\partial A}{\partial x_k }(x,\nabla u) u_{x_k} \nabla \psi\, dx-
\int_{\Omega} \frac{\partial A}{\partial x_k }(x,\nabla u) \nabla u_{x_k} \psi\, dx
\\&-\int_{\Omega} a_{ij}(x,\nabla u) u_{x_j x_k} u_{x_k x_i}\psi\, dx
-
\int_{\Omega} a_{ij}(x,\nabla u) u_{x_j x_k} u_{x_k}\psi_{x_i}\, dx
\\
\leq &-\int_{\Omega} \frac{\partial A}{\partial x_k }(x,\nabla u) u_{x_k} \nabla \psi\, dx-
\int_{\Omega} \frac{\partial A}{\partial x_k }(x,\nabla u) \nabla u_{x_k} \psi\, dx
-
\int_{\Omega} a_{ij}(x,\nabla u) u_{x_j x_k} u_{x_k}\psi_{x_i}\, dx.
\end{align*}

Observe that $\di v_{x_j}=\frac{\nabla u}{|\nabla u|} \nabla u_{x_j} =
\frac{ u_{x_k}}{|\nabla u|} u_{x_k x_j}$.
Taking the sum over $k$ in the last inequality, using Remark \ref{nodiver} and replacing by $v$, we have
\begin{align*}
-\int_{\Omega} &a_{ij}(x,\nabla u) |\nabla u| v_{x_j} \psi_{x_i}\, dx\geq
\int_{\Omega} \frac{\partial A}{\partial x_k }(x,\nabla u) u_{x_k} \nabla \psi\, dx+
\int_{\Omega} \frac{\partial A}{\partial x_k }(x,\nabla u) \nabla u_{x_k} \psi\, dx\\
&=
\int_{\Omega} |\nabla u|^{p-2}\log|\nabla u| p_{x_k} u_{x_i} u_{x_k} \psi_{x_i}\, dx+
\int_{\Omega} |\nabla u|^{p-2}\log|\nabla u| p_{x_k} u_{x_i} u_{x_k x_i} \psi\, dx\\&=
\int_{\Omega} |\nabla u|^{p-2}\log|\nabla u| \langle \nabla u,\nabla p\rangle \langle \nabla u,\nabla \psi\rangle \, dx+
\int_{\Omega} |\nabla u|^{p-2}\log|\nabla u| |\nabla u| v_{x_k} p_{x_k} \psi\, dx
\\&=
\int_{\Omega} |\nabla u|^{p-2}\log|\nabla u| \langle \nabla u,\nabla p\rangle \langle \nabla u,\nabla \psi\rangle \, dx+
\int_{\Omega} |\nabla u|^{p-2}\log|\nabla u| |\nabla u| \langle \nabla p, \nabla v\rangle \psi\, dx.
\end{align*}

By our election of $D, B$ and $H$ we have,
\begin{align*}
-
\int_{\Omega} D_{ij}(x,\nabla u) & v_{x_j} \psi_{x_i}\, dx\geq
\int_{\Omega} H\nabla  \psi \, dx+
\int_{\Omega} B \nabla v\psi\, dx
 .
\end{align*}
\end{proof}

\def\cprime{$'$} \def\ocirc#1{\ifmmode\setbox0=\hbox{$#1$}\dimen0=\ht0
  \advance\dimen0 by1pt\rlap{\hbox to\wd0{\hss\raise\dimen0
  \hbox{\hskip.2em$\scriptscriptstyle\circ$}\hss}}#1\else {\accent"17 #1}\fi}
\providecommand{\bysame}{\leavevmode\hbox to3em{\hrulefill}\thinspace}
\providecommand{\MR}{\relax\ifhmode\unskip\space\fi MR }
\providecommand{\MRhref}[2]{%
  \href{http://www.ams.org/mathscinet-getitem?mr=#1}{#2}
}
\providecommand{\href}[2]{#2}

\end{document}